\documentclass[11pt,a4paper]{article}
\usepackage{amsmath,amssymb,amsthm,amscd,mathrsfs}
\usepackage{indentfirst}
\usepackage[top=25mm, bottom=30mm, left=30mm, right=30mm]{geometry}
\usepackage{color}
\usepackage{comment}

\newtheorem{df}{\bf Definition}[section]
\newtheorem{thm}[df]{\bf Theorem}
\newtheorem{lem}[df]{\bf Lemma}
\newtheorem{cor}[df]{\bf Corollary}
\newtheorem{prop}[df]{\bf Proposition}
\newtheorem{rem}[df]{\bf Remark}

\newtheorem*{claim}{\bf Claim}

\newtheorem{thmA}{\bf Theorem}
\newtheorem{corA}[thmA]{\bf Corollary}

\newcommand{\R}{\mathbb{R}}
\newcommand{\C}{\mathbb{C}}
\newcommand{\Z}{\mathbb{Z}}

\newcommand{\Q}{\mathbb{Q}}
\newcommand{\N}{\mathbb{N}}

\newcommand{\K}{\mathbb{K}}
\newcommand{\M}{\mathbb{M}}

\newcommand{\ri}{\mathrm{i}}
\newcommand{\actson}{\curvearrowright}

\newcommand{\Ad}{\operatorname{Ad}}
\newcommand{\id}{\text{\rm id}}

\newcommand{\Aut}{\operatorname{Aut}}

\newcommand{\Tr}{\mathord{\text{\rm Tr}}}

\title{\bf Cocycle perturbations and ergodicity for actions on type III factors}
\author{Yusuke Isono\thanks{Research Institute for Mathematical Sciences, Kyoto University, 606-8502, Kyoto, Japan}}
\date{}

\begin{document}

\maketitle

\renewcommand{\thefootnote}{}% 番号書式を"空"にする
\footnote[0]{E-mail: \texttt{isono@kurims.kyoto-u.ac.jp}}
\footnote[0]{YI is supported by JSPS KAKENHI Grant Number 24K06759.}
\renewcommand{\thefootnote}{\arabic{footnote}}% 元(算用数字)に戻す

\begin{abstract}
We study cocycle perturbations of state preserving actions on type $\mathrm{III}_1$ factors. 
Extending the theorem of Marrakchi and Vaes for type $\mathrm{II}_1$ factors, we show that a 
state preserving outer $\mathbb Z$-action on a type $\mathrm{III}_1$ factor with trivial 
bicentralizer admits a unitary cocycle whose perturbation becomes an ergodic action. 
This partially answers a question of Marrakchi and Vaes.
A major difference from the type $\mathrm{II}$ case is that the modular automorphism group 
naturally appears as part of the action, making the construction of the required cocycle 
more delicate.
\end{abstract}

\section{Introduction}

Cocycle perturbations of group actions provide a useful method for modifying dynamical
properties while preserving the underlying von Neumann algebraic structure.
In the case of type $\mathrm{II}_1$ factors, Marrakchi and Vaes \cite{MV23} proved that for every
outer action of a countable amenable group, one can find a 1-cocycle such that the 
perturbed action becomes ergodic, and moreover that such cocycles form a dense $G_\delta$
subset of the cocycle space. 
We refer to such a cocycle as an \emph{ergodic cocycle}. 
This settles the ergodicity problem for cocycle perturbations in the type $\mathrm{II}_1$ setting, 
and in particular implies that for a single automorphism, the existence of an ergodic cocycle 
is equivalent to the outerness of all its nonzero powers.

The purpose of this paper is to extend the above framework to the \emph{state preserving} 
setting for type $\mathrm{III}$ factors, where state preserving means that the action 
leaves a faithful normal state invariant.
While many concrete examples of ergodic actions on type $\mathrm{III}$ factors are known, 
it has remained open whether an outer action admits an ergodic cocycle.
The state preserving assumption ensures that the action extends to an isometry on the 
GNS Hilbert space, which underlies the topological arguments used in \cite{MV23} and 
in the present work. 
We also note that Marrakchi and Vaes \cite[Open problem~4.6]{MV23} explicitly asked whether their 
results could extend to type $\mathrm{III}$ factors and suggested that the state preserving 
case might be tractable.

Our main theorem, stated below, gives an affirmative answer for type $\mathrm{III}_1$ factors under the assumption of \emph{trivial bicentralizer}, see Section~\ref{Preliminaries}. 
Very recently, Houdayer and Marrakchi announced a solution to Connes' bicentralizer problem \cite{HM26}, which implies that this assumption is automatic for type $\mathrm{III}_1$ factors with separable predual. We nevertheless keep the trivial bicentralizer assumption explicit whenever it is used, so that the statements reflect the precise hypotheses of our arguments. Consequently, in view of their result, the existence of ergodic cocycles for single automorphisms in the state preserving type $\mathrm{III}_1$ setting is settled in full generality.

\begin{thmA}\label{thmA}
Let $M$ be a type $\mathrm{III}_1$ factor with separable predual and with trivial bicentralizer. Let $\theta \in \Aut(M)$ be a state preserving automorphism. Then the following conditions are equivalent:
\begin{enumerate}
	\item $\theta$ has infinite order in $\mathrm{Out}(M)$;
	\item there exists $u \in \mathcal U(M)$ such that $\Ad(u)\circ\theta$ is state preserving and ergodic.
\end{enumerate}
\end{thmA}

When $M$ is amenable, the state preserving assumption can be removed by using the classification of automorphisms of the amenable type $\mathrm{III}_1$ factor \cite{KST89}. 
This yields the following corollary.

\begin{corA}\label{corB}
Every automorphism of the amenable type $\mathrm{III}_1$ factor $R_\infty$ that has infinite order in $\mathrm{Out}(R_\infty)$ admits an ergodic cocycle.
\end{corA}

For a von Neumann algebra $M$ with a faithful state $\varphi \in M_\ast$, define
\[
\mathrm{Mod}(M)
:= \{\, \Ad(u)\circ\sigma_t^{\varphi} \in \Aut(M) \mid u \in \mathcal U(M),\ t \in \mathbb R \,\}.
\]
This is a normal subgroup of $\Aut(M)$ containing $\mathrm{Int}(M)$. 
For a discrete group action $\alpha \colon \Gamma \to \Aut(M)$, we define the normal subgroup
\[
\Gamma_{\mathrm{mod}} := \alpha^{-1}(\mathrm{Mod}(M)) \le \Gamma.
\]
For outer actions on type $\mathrm{II}$ factors, one always has $\Gamma_{\mathrm{mod}}=\{e\}$. 
One of the major differences in the type $\mathrm{III}$ setting is that $\Gamma_{\mathrm{mod}}$
may be nontrivial.
The proof of Theorem~\ref{thmA} will be divided according to whether 
$\Gamma_{\mathrm{mod}}$ is trivial or not.

\bigskip
\noindent
{\bf Case 1: $\Gamma_{\mathrm{mod}}$ is trivial.}
In the case where $\Gamma_{\mathrm{mod}}$ is trivial, we prove the following theorem, 
which is a direct generalization of \cite{MV23} to the type III setting.

\begin{thmA}\label{thmC}
Let $\alpha \colon \Gamma \actson M$ be a state preserving outer action of a countably infinite discrete group $\Gamma$ on a type $\mathrm{III}_1$ factor with separable predual. 
Assume that $\Gamma$ is amenable and that $M$ has trivial bicentralizer. 
If $\Gamma_{\mathrm{mod}}$ is trivial, then there exists an $\alpha$-cocycle 
$u \colon \Gamma \to \mathcal U(M)$ such that the perturbed action 
is state preserving and weakly mixing.
\end{thmA}

Up to stabilizing by another type $\mathrm{III}_1$ factor, we obtain the following consequence.

\begin{corA}\label{corD}
Let $\alpha \colon \Gamma \actson M$ be a state preserving outer action of a countably infinite discrete amenable group on a type $\mathrm{III}_1$ factor with separable predual. Let $N$ be another type $\mathrm{III}_1$ factor with separable predual. Then the action $\alpha \,\overline\otimes\, \id_N \colon \Gamma \actson M \,\overline\otimes\, N$ admits a cocycle perturbation that is state preserving and weakly mixing.
\end{corA}

The original proof of \cite{MV23} in the case of $\mathrm{II}_1$ factors relies heavily on Popa's free independence result \cite{Po95} (and its variant in \cite{PSV18}). This free independence theorem has been generalized to the setting of type III factors under appropriate hypotheses in \cite{HI14}. Therefore, if a state preserving action $\alpha\colon \Gamma \actson (M,\varphi)$ in Theorem~\ref{thmC} further satisfies $M_\varphi' \cap (M\rtimes_\alpha\Gamma)=\mathbb C$, then by making use of \cite{HI14}, the proof becomes a rather direct adaptation of the argument in \cite{MV23}. 
This special case is stated more precisely in Proposition \ref{prop-case3}.

However, the theorem also covers the case where $M_\varphi=\mathbb C$, in which \cite{HI14} cannot be applied directly. The key point of our proof is the Connes--St\o rmer transitivity theorem for type $\mathrm{III}_1$ factors, which ensures that all faithful states become unitarily conjugate at the level of ultrapowers. As a consequence, $(M^\omega)_{\varphi^\omega}$ is sufficiently large for every $\varphi$, and this fact is the key technical ingredient of our proof.

\bigskip
\noindent
{\bf Case 2: $\Gamma_{\mathrm{mod}}$ is nontrivial.}
When $\Gamma_{\mathrm{mod}}$ is nontrivial, it is not clear how to deal with this situation in full generality. We therefore restrict ourselves to the case $\Gamma=\mathbb Z$, that is, to single automorphisms. In this case $\Gamma_{\mathrm{mod}}$ has the form $p\mathbb Z \leq \mathbb Z$ for some $p>0$, and the action restricted to $p\mathbb Z$ agrees with a modular action. By \cite{MV23}, one obtains an ergodic cocycle for $p\mathbb Z$. We then analyze how such a cocycle can be lifted appropriately to the entire group $\mathbb Z$. This step does not appear in \cite{MV23} and requires a new argument, which will be carried out in Theorem~\ref{thm-case2}.
\bigskip

In the final section, we collect several related results, including the case of type $\mathrm{III}_\lambda$ factors and additional properties of actions and cocycles on ultrapowers.

\bigskip
\noindent
{\bf Declaration on the use of AI.}
In preparing the original manuscript, ChatGPT Plus was used
in December 2025 only to improve the English presentation. 
In revising the manuscript in response to the referee's comments, GPT-5.6 Sol with Extra High reasoning was also used to assist in checking some points raised by the referee. The mathematical results and proofs in the original manuscript were obtained without the use of AI tools.

\tableofcontents

\section{Preliminaries}
\label{Preliminaries}

	Let $M$ be a von Neumann algebra and $\varphi\in M_\ast$ a faithful state. The \textit{modular operator, conjugation}, and \textit{action} are denoted by $\Delta_\varphi$, $J_\varphi$, and $\sigma^\varphi$ respectively. The \textit{centralizer algebra} $M_\varphi$ is a fixed point algebra of the modular action. The norm $\|\, \cdot \, \|_\infty$ is the operator norm of $M$, and $\|\, \cdot \, \|_{2,\varphi}$ (or $\| \, \cdot \, \|_{\varphi}$) is the $L^2$-norm by $\varphi$. We use the notation $x\varphi y = \varphi(y\, \cdot \, x)$ for all $x,y\in M$ and $\varphi\in M_\ast$. See e.g.\ \cite{Ta03} for background on these objects.

\subsection*{Ergodic actions and cocycles}

Let $\Gamma \curvearrowright^\alpha M$ be a discrete group action on a von Neumann algebra. We say that a map $v\colon \Gamma \to M$ is an \textit{$\alpha$-cocycle} if it satisfies $v_{st}=v_s \alpha_s(v_t)$ for all $s,t\in \Gamma$. In this case, we can define the cocycle perturbed action 
	$$\alpha^v\colon \Gamma \to \Aut(M);\quad \alpha^v_g := \Ad(u_g)\circ \alpha_g  \ (g\in \Gamma).$$
If $\Gamma=\Z$ with $\alpha_1= \theta\in \Aut(M)$, then any unitary $u\in \mathcal U(M)$ gives rise to a unique $\alpha$-cocycle $v$ such that $v_1 = u$. Then, the $\Z$-action given by $\Ad(u)\circ \theta$ coincides with $\alpha^v$.

We say that a discrete group action $\alpha \colon \Gamma \actson M$ is \emph{ergodic} if the fixed point algebra $M^\alpha$ is trivial.
It is called \emph{state preserving} if there exists a faithful state $\varphi\in M_\ast$ such that $\varphi\circ \alpha_g=\varphi$ for every $g\in \Gamma$. 
In this case, we write $\alpha \colon \Gamma \actson (M,\varphi)$. We say that a state preserving action $\alpha \colon \Gamma \actson (M,\varphi)$ is \emph{weakly mixing}  (with respect to $\varphi$) if for any $\varepsilon>0$ and any finite subset $E\subset M$, there exists $g\in \Gamma$ such that 
	$$ \sum_{x,y\in E}| \varphi(\alpha_g(x)y) - \varphi(x)\varphi(y) | < \varepsilon.$$
Equivalently, the unitary representation $U^\alpha \colon \Gamma \actson L^2(M,\varphi)\ominus \C 1_M$ induced by $\alpha$ is weakly mixing. Note that any weakly mixing action is ergodic. For a single automorphism $\theta\in\Aut(M)$, we use the same terminology by regarding it as a $\mathbb Z$-action via $\theta^n$ for $n\in\mathbb Z$.

\begin{thm}[\cite{MV23}]\label{thm-MV23}
Let $M$ be a type $\mathrm{III}_1$ factor with separable predual.  Then there exists a faithful state $\varphi \in M_\ast$ such that for every $T>0$,  the automorphism $\sigma^\varphi_{T}$ is weakly mixing with respect to $\varphi$.
\end{thm}
\begin{proof}
By \cite[Theorem~A and Lemma~2.1]{MV23}, we choose $\varphi$ such that $\sigma^\varphi$ is weakly mixing with respect to $\varphi$ as an $\R$-action. It is equivalent to that $\{\Delta_{\varphi}^{\ri t}\}_{t\in \R }'\cap \K(H) = 0$, where $H :=L^2(M,\varphi)\ominus \C 1_M$. Since $\mathbb R / T\mathbb Z$ is compact, there exists a faithful conditional expectation from $\{\Delta_{\varphi}^{\ri Tn}\}_{n\in \Z }'\cap \K(H )$ onto $\{\Delta_{\varphi}^{\ri t}\}_{t\in \R }'\cap \K(H) $. We conclude that $\{\Delta_{\varphi}^{\ri Tn}\}_{n\in \Z }'\cap \K(H) = 0$, which means that $\sigma^\varphi|_{T\mathbb Z}$ is weakly mixing.
\end{proof}

\subsection*{Relative bicentralizers}

Let $M$ be a $\sigma$-finite von Neumann algebra and fix a free ultrafilter
$\omega$ on $\mathbb N$.  Consider
\[
\mathcal I_\omega=\{ (x_n)_n\in \ell^\infty(\mathbb N,M)\mid x_n\to 0 
\text{ $\ast$-strongly along }\omega\},
\]
and
\[
\mathcal M^\omega=\{ (x_n)_n\in \ell^\infty(\mathbb N,M)\mid 
(x_n)_n \mathcal I_\omega\subset \mathcal I_\omega,\;
\mathcal I_\omega (x_n)_n\subset \mathcal I_\omega\}.
\]
The ultraproduct von Neumann algebra $M^\omega$ \cite{Oc85} is defined as the
quotient $\mathcal M^\omega / \mathcal I_\omega$, and the class of
$(x_n)_n$ is written $(x_n)_\omega$.
Via constant sequences we regard $M$ as a subalgebra of $M^\omega$, with
the canonical expectation $E_\omega$ given by
$E_\omega((x_n)_\omega)=\sigma\hbox{-weak}\lim_{n\to\omega}x_n$.
Every faithful state $\varphi\in M_\ast$ induces a faithful normal state
$\varphi^\omega=\varphi\circ E_\omega$ on $M^\omega$.
Every discrete group action $\alpha \colon \Gamma \actson M$ extends on $M^\omega$ by $\alpha_g( (x_n)_\omega ) = (\alpha_g(x_n))_\omega$, and we denote it by $\alpha^\omega \colon \Gamma \actson M^\omega$. For more on ultraproducts, we refer the reader to \cite{Oc85,AH12}.

For an inclusion of von Neumann algebras $N\subset M$ with expectation $E_N$, we have a canonical inclusion $N^\omega\subset M^\omega$. Then for a faithful state $\varphi\in N_\ast$,  which is extended on $M$ via $E_N$, we can define $\sigma^{\varphi^\omega}$ that globally preserves $N^\omega$. We then define the \textit{relative bicentralizer algebra} (with respect to $\varphi$) as
	$$ \mathrm{BC}(N\subset M,\varphi) = (N^\omega)_{\varphi^\omega} '\cap M. $$
The definition does not depend on $\omega$. When $N=M$, we write $\mathrm{BC}(M,\varphi)$. We say that a type $\rm III_1$ factor $M$ has \emph{trivial bicentralizer} if $\mathrm{BC}(M,\varphi)=\C$ for some faithful state $\varphi\in M_\ast$. In this case, the bicentralizer is trivial for all faithful state $\varphi\in M_\ast$. See \cite{AHHM18} for more on relative bicentralizers.

The next lemma is a consequence of \cite{Ma23}.

\begin{lem}\label{lem-relative-bicentralizer}
	Let $\Gamma \actson M$ be a state preserving outer action of a countable group on a von Neumann algebra. Assume that:
\begin{itemize}
    \item $\Gamma_{\mathrm{mod}}$ is trivial;
    \item $M$ is a type $\mathrm{III}_1$ factor with separable predual and has trivial bicentralizer.
\end{itemize}
Then we have $\mathrm{BC}(M \subset M\rtimes_\alpha \Gamma, \varphi)=\mathbb C$ for all faithful states $\varphi\in M_\ast$.
\end{lem}

\begin{proof}
	In the second half of the proof of \cite[Theorem~A$(3)\Rightarrow(4)$]{Is23}, it is shown that
	$\mathrm{BC}(M\subset \widetilde{M},\varphi)=\mathbb C$ for
	$\widetilde M=M\rtimes G$ and some faithful state $\varphi\in M_\ast$.
	The same argument applies here. Indeed, replacing the commutativity argument there by the standard Fourier coefficient argument for
	\[
	C_\varphi(M\rtimes_\alpha\Gamma)
	\cong
	C_\varphi(M)\rtimes_{\widetilde\alpha}\Gamma,
	\]
	and using $M'\cap C_\varphi(M)=\mathbb C$, we see that the assumption
	$\Gamma_{\mathrm{mod}}=\{e\}$ gives
	\[
	M'\cap C_\varphi(M\rtimes_\alpha\Gamma)=\mathbb C.
	\]
	The rest of the proof is unchanged, and the conclusion does not depend
	on the choice of $\varphi$.
\end{proof}

\subsection*{Free independence in ultraproducts}

Popa's free independence for ultraproducts was obtained in the framework of type $\mathrm{II}_1$ factors in \cite{Po95}.  Later, a generalization to type $\mathrm{III}$ factors was proved in \cite{HI14}.  The following proposition is a slight refinement of that result.

\begin{prop}\label{prop-free-independence1}
Let $M\subset \widetilde{M}$ be an inclusion of $\sigma$-finite von Neumann algebras with expectation $E$, and assume that $\widetilde M$ is a diffuse factor. Let $X\subset \widetilde{M}^\omega$ be a subset that is separable for the $\sigma$-weak topology.

If there exists a faithful state $\varphi\in \widetilde{M}_\ast$ with $\varphi\circ E=\varphi$ such that $\mathrm{BC}(M\subset \widetilde{M}, \varphi)=\mathbb C$, then there exists $U\in \mathcal U\bigl((M^\omega)_{\varphi^\omega}\bigr)$ such that $X$ and $UXU^\ast$ are $\ast$-free with respect to $\varphi^\omega$.
\end{prop}
\begin{proof}
Set $\mathcal M := M^\omega \subset \widetilde{M}^\omega =: \widetilde{\mathcal M}$
and $\Phi := \varphi^\omega$.  By \cite[Proposition~3.3]{AHHM18} we have
\[
\mathcal M_{\Phi}' \cap \widetilde{\mathcal M}
   \subset \mathrm{BC}(M\subset \widetilde{M},\varphi)^\omega
   = \mathbb C.
\]
Thus the inclusion $\mathcal M_{\Phi} \subset \widetilde{\mathcal M}$ is an irreducible inclusion with expectation. In particular, $\widetilde{\mathcal M}$ is a factor and it is diffuse since it contains $\widetilde M$ with expectation. 
Let $ P\subset \widetilde{\mathcal M}$ be the von Neumann algebra generated by $\sigma^{\Phi}_t(x)$ for $x\in X$ and $t\in \Q$. Then it has separable predual and the inclusion $P \subset \widetilde{\mathcal M}$ is with expectation as it is globally invariant under $\sigma^{\Phi}$.

We apply \cite[Theorem~A]{HI14} to $\mathcal M_{\Phi} \subset \widetilde{\mathcal M}$ and the subalgebra $P \subset \widetilde{\mathcal M}$, and find a unitary $\mathcal V \in \mathcal U\bigl((\mathcal M_{\Phi})^\omega\bigr)$ such that $ P$ and $\mathcal V P \mathcal V^\ast$ are $\ast$-free with respect to $\Phi^\omega$. 
Let $ Y\subset P$ be a $\sigma$-weakly dense countable subset, and observe that the $\ast$-freeness of $ P$ and $\mathcal VP \mathcal V^\ast$ is equivalent to that of $ Y$ and $\mathcal V Y \mathcal V^\ast$. 
We will construct, by a standard diagonal argument, a unitary $U\in \mathcal U( \mathcal M_\Phi)$ such that $ Y$ and $U Y U^\ast$ are $\ast$-free with respect to $\Phi^\omega|_{\widetilde{\mathcal M}} =\Phi= \varphi^\omega$. Then $ P$ and $UP U^\ast$ are $\ast$-free, and we will get the conclusion.

To distinguish the two ultrapower levels, we use the notation such as 
\[
\mathcal V\in \widetilde{\mathcal M}^\omega=(\widetilde{M}^\omega)^\omega,\quad V\in \widetilde{\mathcal M} = \widetilde{M}^\omega, \quad \text{and}\quad v\in \widetilde M.
\]
Choose a representing unitary sequence $(V_n)_n$ of $\mathcal V$, where $V_n \in \mathcal U(\mathcal M_{\Phi})$. 
Write $Y=\{Y_1,Y_2,\ldots\}$. 
Enumerate the alternating centered moments characterizing the $*$-freeness of $Y$
and $\mathcal VY\mathcal V^*$ by noncommutative $*$-polynomials
\[
P_k(Z_1,\ldots,Z_{r_k+2}),
\qquad
k,\ r_k\in\mathbb N,
\]
so that the $*$-freeness of $Y$
and $\mathcal VY\mathcal V^*$ with respect to $\Phi^\omega$ is equivalent to
\[
\Phi^\omega\bigl(
P_k(Y_1,\ldots,Y_{r_k},\mathcal V,\mathcal V^*)
\bigr)=0,
\qquad
k\in\mathbb N.
\]
For example, a typical condition is
\[
\Phi^\omega\bigl(
A^\circ \mathcal VB^\circ \mathcal V^*C^\circ \mathcal VD^\circ \mathcal V^*
\bigr)=0,
\qquad
A,B,C,D\in Y,
\]
where $z^\circ:=z-\Phi(z)1$, which is obtained from one of the polynomials $P_k$ above.
Then for every $k\in\mathbb N$, we may
choose $n_k\in\mathbb N$ such that
\[
\left|
\Phi\bigl(
P_j(Y_1,\ldots,Y_{r_j},V_{n_k},V_{n_k}^*)
\bigr)
\right|
<
\frac{1}{k},
\qquad
1\leq j\leq k.
\]
Choose representing sequences $(v_k^{(m)})_m$ and $(y_j^{(m)})_m$ of each $V_{n_k}$ and $Y_j$ respectively, where $v_k^{(m)}\in\mathcal U(M)$ and $y_j^{(m)}\in \widetilde M$. 
Since $V_{n_k}\in\mathcal M_\Phi = (M^\omega)_{\varphi^\omega}$, we have 
\[
\lim_{m\to\omega}
\|v_k^{(m)}\varphi-\varphi v_k^{(m)}\|
=
0,
\]
and for every $1\leq j\leq k$, 
\[
\left|
\Phi\bigl(
P_j(Y_1,\ldots,Y_{r_j},V_{n_k},V_{n_k}^*)
\bigr)
\right|
=
\lim_{m\to\omega}
\left|
\varphi\bigl(
P_j(
y_1^{(m)},\ldots,y_{r_j}^{(m)},
v_k^{(m)},(v_k^{(m)})^*
)
\bigr)
\right|
<
\frac{1}{k}.
\]
It follows that there exists $B_k\in\omega$ such that, for every
$m\in B_k$,
$\|v_k^{(m)}\varphi-\varphi v_k^{(m)}\|
<
\frac{1}{k}
$ and 
\[
\left|
\varphi\bigl(
P_j(
y_1^{(m)},\ldots,y_{r_j}^{(m)},
v_k^{(m)},(v_k^{(m)})^*
)
\bigr)
\right|
<
\frac{1}{k},
\qquad
1\leq j\leq k.
\]
Put $A_k:=
\bigcap_{\ell=1}^k B_\ell
\cap
\{m\in\mathbb N\mid m\geq k\}$. We may then assume that
\[
A_{k+1}\subset A_k,
\qquad
A_k\subset\{m\in\mathbb N\mid m\geq k\}.
\]
Put $A_0=\mathbb N$ and $v_0^{(m)}=1$ for every $m\in\mathbb N$.
For each $m\in\mathbb N$, define
\[
s(m)
:=
\max\{k\in \mathbb N\cup \{0\}\mid m\in A_k\},
\qquad
u^{(m)}:=v_{s(m)}^{(m)}.
\]
This maximum is well defined since $m\in A_k$ implies $k\leq m$. 
Note that $\lim_{m\to\omega}s(m)=\infty$ holds, since $m\in A_k$ implies $s(m)\geq k$. We show that $(u^{(m)})_m$ satisfies the desired condition.

Take any $k$ and $m\in A_k$. Then, we
have $s(m)\geq k$ and $m\in A_{s(m)}$. Therefore,
\[
\left|
\varphi\bigl(
P_j(
y_1^{(m)},\ldots,y_{r_j}^{(m)},
u^{(m)},(u^{(m)})^*
)
\bigr)
\right|
<
\frac{1}{s(m)}
\leq
\frac{1}{k},
\qquad
j=1,\ldots,k.
\]
Since $A_k\in\omega$, this implies
\[
\lim_{m\to\omega}
\varphi\bigl(
P_j(
y_1^{(m)},\ldots,y_{r_j}^{(m)},
u^{(m)},(u^{(m)})^*
)
\bigr)
=
0,
\qquad
j\in\mathbb N.
\]
Similarly, we have
\[
\lim_{m\to\omega}
\|u^{(m)}\varphi-\varphi u^{(m)}\|
=
0,
\]
and $U:=(u^{(m)})_\omega$ defines an element in
$\mathcal U((M^\omega)_{\varphi^\omega})$. We obtain 
\[
\varphi^\omega\bigl(
P_j(Y_1,\ldots,Y_{r_j},U,U^*)
\bigr)
=
0,
\qquad
j\in\mathbb N.
\]
Since $\Phi^\omega|_{\widetilde{\mathcal M}}=\Phi=\varphi^\omega$, we obtain that $Y$ and $UYU^*$ are $*$-free with respect to $\varphi^\omega$. This completes the proof.
\end{proof}
\begin{rem}\upshape\label{rem-free-independence1}
In the setting of Proposition~\ref{prop-free-independence1}, if we further assume that $M_\varphi'\cap \widetilde M=\C$, then the resulting unitary can be chosen in $\mathcal U((M_\varphi)^\omega)$. Indeed, one can construct the von Neumann algebra $P$ from $X$ as in the proof above. Then \cite[Theorem~A]{HI14} yields the desired unitary. This fact is already contained in the proof of \cite[Theorem~A]{HI14}. 
More precisely, it is established first, and the free independence conclusion is then deduced from it. This stronger form will be used only in Section~\ref{Further results}.
\end{rem}

\section{Lemmas on ultraproducts}

The aim of this section is to prove the following technical lemma.

\begin{lem}\label{lem-cocycle3.6}
	Let $\alpha\colon \Gamma \actson (M,\varphi)$ be a state preserving outer action of a countably infinite discrete group on a diffuse factor with separable predual and with a faithful normal state.  
Assume that $\Gamma$ is amenable and that $\mathrm{BC}(M\subset M\rtimes \Gamma,\varphi)=\C$. 
Then for every $\varepsilon>0$ and every finite subset $F\subset \Gamma$, there exist $h\in \Gamma$ and $U\in \mathcal U(M^\omega)$ such that:
\begin{itemize}
	\item $\varphi^\omega(U)=0$, \quad $\|U\varphi^\omega-\varphi^\omega U\|<\varepsilon$,\quad and \quad $\|\alpha_g^\omega(U)-U\|_{\varphi^\omega}<\varepsilon$ for all $g\in F$;
	\item $M$, $U$, and $\alpha_h^\omega(U)$ are $\ast$-free with respect to $\varphi^\omega$.
\end{itemize}
\end{lem}

We need several lemmas. The next lemma corresponds to \cite[Lemma 4.2]{PSV18} and identifies a structure in the ultraproduct that is isomorphic to a free Bernoulli action.

\begin{lem}\label{lem-free-independence2}
Keep the same setting as in Lemma \ref{lem-cocycle3.6}.  
Then there exists $V\in \mathcal U((M^\omega)_{\varphi^\omega})$ such that $M$ and the family $\{\alpha_g^\omega(VMV^*)\}_{g\in\Gamma}$ are mutually $\ast$-free inside $M^\omega$ with respect to $\varphi^\omega$.  
Moreover, if we set
\[
P := \mathrm W^*\{\alpha_g^\omega(VMV^*)\}_{g\in\Gamma},
\]
then we have a natural identification
\[
 (M\vee P,\varphi^\omega)\simeq (M,\varphi)\ast\Bigl(\ast_{g\in\Gamma}(M,\varphi)\Bigr),
\]
and the restriction of $\alpha^\omega$ to this subalgebra coincides with the free product action of $\alpha$ and the free Bernoulli shift action.
\end{lem}

\begin{proof}
	We apply Proposition \ref{prop-free-independence1} to the inclusion $M\subset \widetilde{M}:=M\rtimes \Gamma$ and $X:=M\rtimes \Gamma$, and find a unitary $V\in \mathcal U((M^\omega)_{\varphi^\omega})$ such that $M\rtimes\Gamma$ and $V(M\rtimes\Gamma)V^*$ are $\ast$-free in $(M\rtimes\Gamma)^\omega$ with respect to $\varphi^\omega$.  
Using the relation $\alpha_g^\omega(VxV^*)=\lambda_gVxV^*\lambda_g^*$ for $x\in M$ and $g\in \Gamma$ together with the $\ast$-freeness, it is easy to  see that $M$ and $\{\alpha_g^\omega(VMV^*)\}_{g\in\Gamma}$ are mutually $\ast$-free inside $M^\omega$ with respect to $\varphi^\omega$.

Since $\varphi^\omega\circ\alpha_g^\omega\circ\Ad(V)|_M=\varphi$, there exists a state preserving isomorphism
\[
 (M,\varphi)\ni x \longmapsto \alpha_g^\omega(VxV^*) \in (\alpha_g^\omega(VMV^*),\varphi^\omega).
\]
It follows that $(P,\varphi^\omega)$ is isomorphic to the free product $\ast_{\Gamma}(M,\varphi)$. It is obvious that $\alpha^\omega$ restricts to the free Bernoulli action on $P$.
\end{proof}

The next lemma concerns a free Bernoulli action.

\begin{lem}\label{lem-amenable-strongly-ergodic}
Let $\Gamma$ be an amenable countable discrete group, and let $(N,\psi)$ be a diffuse von Neumann algebra with a faithful normal state.  
Consider the free Bernoulli shift
\[
 \alpha\colon \Gamma \actson (P,\varphi):=\ast_{g\in \Gamma}(N,\psi).
\]
Then for every $\varepsilon>0$ and finite subset $F\subset \Gamma$, there exist a finite subset $J\subset \Gamma$ and a unitary $u\in \ast_{g\in J}(N,\psi)$ such that
	\[
	 \varphi(u)=0,\quad \|\alpha_g(u)-u\|_\varphi<\varepsilon\ (g\in F),\quad \text{and}\quad \|u\varphi-\varphi u\|<\varepsilon.
	\]
\end{lem}

\begin{proof}
The case where $\psi$ is a trace was already treated in the proof of \cite[Theorem B]{MV23}.  For the reader's convenience, we briefly recall the argument in the tracial case.

Assume that $\psi$ is a trace.  
Since $(P^\omega)^{\alpha^\omega}$ is diffuse by \cite[Theorem 4.1]{PSV18}, there exists a projection $p\in P$ such that $\varphi(p)=1/2$ and $\|\alpha_g(p)-p\|_{\varphi}<\varepsilon$ for all $g\in F$.  
For each subset $J\subset \Gamma$, put $P_J:=\ast_{g\in J}(N,\psi)$.  
As the algebras $P_J$ for finite $J\subset\Gamma$ are von Neumann subalgebras whose union is $\ast$-strongly dense in the finite von Neumann algebra $P$, we can find a finite subset $J\subset\Gamma$ and a projection $q\in P_J$ such that $\|p-q\|_\varphi$ is sufficiently small. Since $P_J$ is diffuse, we can replace $p$ by $q$, and we may assume $p$ is contained in $P_J$.  Then $u:=2p-1$ is a unitary element satisfying the desired properties. 

We next see the general case.  By Lemma \ref{lem-centralizer-diffuse} below, $(N^\omega)_{\psi^\omega}$ is diffuse.  
Consider the free Bernoulli action on the finite von Neumann algebra
\[
 \beta\colon \Gamma \actson \ast_{g\in \Gamma}\bigl((N^\omega)_{\psi^\omega},\psi^\omega\bigr)=:(\mathcal N,\Psi).
\]
Since $\psi^\omega$ is a trace on $(N^\omega)_{\psi^\omega}$, the previous argument applies.  
Thus there exist a finite subset $J\subset\Gamma$ and a projection
$q\in\ast_{g\in J}\bigl((N^\omega)_{\psi^\omega},\psi^\omega\bigr)$ such that
$\Psi(q)=1/2$ and $\|\beta_g(q)-q\|_\Psi<\varepsilon$ for all $g\in F$.

Write $P=\ast_{g\in\Gamma}(N_g,\psi_g)$ with $(N_g,\psi_g)=(N,\psi)$.  
Then each inclusion $N_g\subset P$ extends to $N_g^\omega\subset P^\omega$ with $\varphi^\omega$-preserving expectation.  
Since $\{N_g^\omega\}_{g\in\Gamma}$ are $\ast$-free with respect to $\varphi^\omega$, we have a natural inclusion with $\varphi^\omega$-preserving expectation
\[
 \ast_{g\in J}\bigl((N^\omega)_{\psi^\omega},\psi^\omega\bigr)
 \subset \bigl(\ast_{g\in J}(N,\psi)\bigr)_{\varphi^\omega}
 \subset (P^\omega)_{\varphi^\omega}.
\]
Hence we may regard $q$ as an element of $(P^\omega)_{\varphi^\omega}$.  
Since $\beta$ is the restriction of $\alpha^\omega$, by choosing a representing sequence of $q$ in $P^\omega$, we obtain a projection $p\in \ast_{g\in J}(N,\psi)$ such that
\[
 |\varphi(p)-1/2|<\varepsilon,\qquad
 \|\alpha_g(p)-p\|_\varphi<\varepsilon\ (g\in F),\qquad
 \|p\varphi-\varphi p\|<\varepsilon.
\]
As $\ast_{g\in J}(N,\psi)$ is diffuse, by replacing $\varepsilon$ with a smaller one if necessary,  we may assume that $\varphi(p)=1/2$.  
Setting $u:=2p-1$, we obtain the desired conclusion.
\end{proof}
\begin{rem}\upshape\label{rem-free-Bernoulli-diffuse}
In the setting of Lemma~\ref{lem-amenable-strongly-ergodic}, we can see that 
$(P^\omega)^{\alpha^\omega}_{\varphi^\omega}$ is diffuse as follows.  
For the action $\beta$ considered in the proof above, we may apply 
\cite[Theorem~4.1]{PSV18}, hence $(\mathcal N^\omega)^{\beta^\omega}$ is diffuse. It contains a Haar unitary.  Representing such a unitary by a sequence, for any $\varepsilon>0$, $n\in\mathbb N$, and any finite set $F\subset \Gamma$, there exists a unitary $U\in\mathcal U(\mathcal N)$ such that
\[
  |\Psi(U^k)|<\varepsilon \quad (0<|k|\le n),
  \qquad 
  \|\beta_g(U)-U\|_{2,\Psi}<\varepsilon .
\]
Then viewing $U$ as an element of $(P^\omega)_{\varphi^\omega}$ and choosing a representing sequence, there exists $u\in\mathcal U(P)$ satisfying
\[
  |\varphi(u^k)|<\varepsilon \quad (0<|k|\le n),\qquad
  \|\alpha_g(u)-u\|_{2,\varphi}<\varepsilon,\qquad
  \|u\varphi-\varphi u\|<\varepsilon .
\]
It follows that $(P^\omega)^{\alpha^\omega}_{\varphi^\omega}$ contains a Haar unitary.
\end{rem}

\begin{lem}\label{lem-centralizer-diffuse}
Let $M$ be a diffuse $\sigma$-finite von Neumann algebra.  
Then for every faithful state $\varphi\in M_\ast$, the centralizer $(M^\omega)_{\varphi^\omega}$ is diffuse.
\end{lem}

\begin{proof}
This lemma is well known to experts, see the proof of \cite[Lemma 3.6]{Ma18} for the reference. If $M_\varphi$ is diffuse, then since the inclusion $M_\varphi \subset (M^\omega)_{\varphi^\omega}$ is with expectation, it follows that $(M^\omega)_{\varphi^\omega}$ is also diffuse. Suppose next that $M_\varphi$ contains a minimal projection $p$. Then $(pMp)_\varphi = pM_\varphi p = \C p$, but $pMp\neq \C p$ as $M$ is diffuse. 
Hence $pMp$ is a type ${\rm III}_1$ factor \cite[Proof of Theorem 3]{Lo78}, and consequently $(pMp)^\omega_{\varphi^\omega}$ is a type ${\rm II}_1$ factor.

For the general case, let $z$ be the central projection of $M_\varphi$ such that $M_\varphi z$ is atomic and $M_\varphi z^\perp$ is diffuse. Write $z=\sum_i p_i$ as a sum of minimal projections in $M_\varphi z$.  From the argument above, both $z^\perp (M^\omega)_{\varphi^\omega} z^\perp$ and each $p_i (M^\omega)_{\varphi^\omega} p_i$ are diffuse.  Since $1=z^\perp+\sum_i p_i$ is a decomposition into orthogonal projections, it follows that $(M^\omega)_{\varphi^\omega}$ is diffuse.
\end{proof}

We are now ready to prove Lemma \ref{lem-cocycle3.6}.

\begin{proof}[{\bf Proof of Lemma \ref{lem-cocycle3.6}}]
Using Lemma \ref{lem-free-independence2}, choose $V\in \mathcal U((M^\omega)_{\varphi^\omega})$ and define $P$ as in the lemma. For each finite subset $J\subset \Gamma$, let $P_J\subset M^\omega$ denote the von Neumann algebra generated by $\{\alpha_g^\omega(VMV^*)\}_{g\in J}$. Since $\Gamma$ is amenable, we can apply Lemma \ref{lem-amenable-strongly-ergodic} to $P$. Thus for any $\varepsilon>0$, there exist a finite subset $J\subset \Gamma$ and a unitary $U\in \mathcal U(P_J)$ such that
\[
 \varphi^\omega(U)=0,\quad \|\alpha_g^\omega(U)-U\|_{\varphi^\omega}<\varepsilon\ (\forall g\in F),\quad 
 \|(U\varphi^\omega-\varphi^\omega U)|_P\|<\varepsilon.
\]
Since $\Gamma$ is infinite, we can choose $h\in\Gamma$ such that $hJ\cap J=\emptyset$.  
Then $\alpha_h^\omega(U)\in \mathcal U(\alpha_h^\omega(P_J))=\mathcal U(P_{hJ})$, and hence $U$ and $\alpha_h^\omega(U)$ are $\ast$-free from each other.  
Since $M$ is $\ast$-free from $P$, we conclude that $M$, $U$, and $\alpha_h^\omega(U)$ are mutually $\ast$-free.

Finally, since $\alpha$ preserves $\varphi$, the subalgebra $P\subset M^\omega$ admits a $\varphi^\omega$-preserving conditional expectation.  
Therefore,
\[
 \|U\varphi^\omega-\varphi^\omega U\|=\|(U\varphi^\omega-\varphi^\omega U)|_P\|<\varepsilon,
\]
which establishes the last inequality.  
This completes the proof.
\end{proof}

\section{Cocycles for state preserving actions}

Let $\alpha\colon \Gamma \actson M$ be an action of a discrete group on a
$\sigma$-finite von Neumann algebra. We denote by $S_{\mathrm f}(M)\subset M_\ast$ the set of all faithful normal states on $M$, and by $\mathcal C(\alpha)$ the set of all $\alpha$-cocycles. We equip $S_{\mathrm f}(M)$ with the norm topology from $M_\ast$, and $\mathcal C(\alpha)$ with the pointwise $\sigma$-$\ast$strong topology. If $M$ has separable predual and $\Gamma$ is countable, then both $S_{\mathrm f}(M)$ and $\mathcal C(\alpha)$ are Polish spaces (see \cite[Lemma~2.6]{MV23}).

Suppose that $\alpha$ is a state preserving action. Then we define
\[
 \mathcal C_{\mathrm{state}}(\alpha)
 \subset \mathcal C(\alpha)\times S_{\mathrm f}(M)
\]
to be the set of all pairs $(v,\psi)$ such that the perturbed action $\Ad(v)\circ \alpha = \alpha^v$ preserves $\psi$.  It is easy to see that $\mathcal C_{\mathrm{state}}(\alpha)$ is closed in the product topology on $\mathcal C(\alpha)\times S_{\mathrm f}(M)$.  
In particular, if $M$ has separable predual and $\Gamma$ is countable, $\mathcal C_{\mathrm{state}}(\alpha)$ is a Polish space.

We further define the following subsets:
\begin{align*}
 \mathcal C_{\mathrm{erg}}(\alpha)
 &:=\{(v,\psi)\in\mathcal C_{\mathrm{state}}(\alpha)
    \mid \alpha^v \text{ is ergodic}\},\\
 \mathcal C_{\mathrm{wm}}(\alpha)
 &:=\{(v,\psi)\in\mathcal C_{\mathrm{state}}(\alpha)
    \mid \alpha^v \text{ is weakly mixing with respect to }\psi\}.
\end{align*}
It is clear that $\mathcal C_{\mathrm{wm}}(\alpha)\subset
\mathcal C_{\mathrm{erg}}(\alpha)$.

\begin{rem}\upshape
	For each $(v,\psi)\in\mathcal C_{\mathrm{erg}}(\alpha)$, the state $\psi$ is uniquely determined by $v$. Indeed, if $\alpha\colon \Gamma\actson (M,\varphi)$ is a state preserving ergodic action and if there is another invariant state $\varphi'\in M_\ast$, then $\varphi'=\varphi$. This follows from the fact that the Connes cocycle between
$\varphi'$ and $\varphi$ must be scalar valued by the ergodicity.
\end{rem}

By the next lemma, both $\mathcal C_{\mathrm{erg}}(\alpha)$ and
$\mathcal C_{\mathrm{wm}}(\alpha)$ can be expressed as countable
intersections of open sets.

\begin{lem}\label{lem-cocycle2.5}
Under the above setting, for each $x\in M$, finite subset
$E\subset M$, and $\varepsilon>0$, define
\begin{align*}
 \mathcal U(x)
 &:=\{(v,\psi)\in\mathcal C_{\mathrm{state}}(\alpha)\mid
     \exists h\in\Gamma\ \text{s.t.}\
     \|\alpha^v_h(x)-x\|_\psi>\|x-\psi(x)\|_\psi\},\\
 \mathcal V(E,\varepsilon)
 &:=\{(v,\psi)\in\mathcal C_{\mathrm{state}}(\alpha)\mid
     \exists h\in\Gamma\ \text{s.t.}\
     \sum_{a,b\in E}|
       \psi(\alpha^v_h(a)b)-\psi(a)\psi(b)
     |<\varepsilon\}.
\end{align*}
Then the following hold:
\begin{enumerate}
 \item $\mathcal U(x)$ and $\mathcal V(E,\varepsilon)$ are open subsets
       of $\mathcal C_{\mathrm{state}}(\alpha)$.
 \item For any $\sigma$-$\ast$strongly dense subset $X\subset M\setminus\C$
       and sequence $\varepsilon_n\to0$, we have
 \[
   \mathcal C_{\mathrm{erg}}(\alpha)
   =\bigcap_{x\in X}\mathcal U(x),\qquad
   \mathcal C_{\mathrm{wm}}(\alpha)
   =\bigcap_{\substack{E\subset X\ \mathrm{finite}\\ n\in\N}}\mathcal V(E,\varepsilon).
 \]
\end{enumerate}
\end{lem}

\begin{proof}
(1) It is easy to see that the complements of
$\mathcal U(x)$ and $\mathcal V(E,\varepsilon)$ are closed.

(2) Suppose $(v,\psi)\in\mathcal C_{\mathrm{erg}}(\alpha)$.  
Then $\beta:=\alpha^v$ is ergodic and $\psi$-preserving, so
$\psi\colon M\to M^\beta=\C$ coincides with the $\psi$-preserving
conditional expectation onto $M^\beta$.  
Hence for any $x\in M$, by \cite[Lemma 5.1]{AHHM18} there exists
$h\in\Gamma$ such that
\[
 \|x-\beta_h(x)\|_\psi>\|x-\psi(x)\|_\psi.
\]
Since $\beta_h=\Ad(v_h)\circ\alpha_h$, we have
$(v,\psi)\in\mathcal U(x)$.

Conversely, suppose $(v,\psi)\notin\mathcal C_{\mathrm{erg}}(\alpha)$. Let $\beta=\alpha^v$. Then $\C\ne M^\beta$, so there exists
$y\in M^\beta$ with $y-\psi(y)\ne0$.  
Choose $z\in X$ and $\varepsilon>0$ such that
\[
 \|y-z\|_\psi\le\varepsilon<\|y-\psi(y)\|_\psi/3.
\]
Then
\[
 \|z-\psi(z)\|_\psi\ge\|y-\psi(y)\|_\psi-\varepsilon>2\varepsilon,
\]
and for any $h\in\Gamma$, since $\psi\circ \beta_h=\psi$,
\[
 \|\beta_h(z)-z\|_\psi
 \le\|\beta_h(z-y)\|_\psi+\|z-y\|_\psi
 \le2\varepsilon<\|z-\psi(z)\|_\psi.
\]
Thus $(v,\psi)\notin\mathcal U(z)$.

The statement for $\mathcal C_{\mathrm{wm}}(\alpha)$ follows easily by definition.
\end{proof}

We translate the result in the previous section into the following form.

\begin{lem}\label{lem-cocycle3.5}
Let $\alpha\colon \Gamma \actson (M,\varphi)$ be a state preserving outer action of a countably infinite discrete group on a diffuse factor with separable predual.  Assume that $\Gamma$ is amenable and that $\mathrm{BC}(M\subset M\rtimes \Gamma,\varphi)=\C$. For any $x\in M\setminus\C$, any finite subset $E\subset M\setminus\C$, and any $\varepsilon>0$, the pair $(1,\varphi)\in\mathcal C_{\mathrm{state}}(\alpha)$ can be approximated by elements of $\mathcal U(x)\cap \mathcal V( E,\varepsilon)$, where $1\in\mathcal C(\alpha)$ is the trivial cocycle.
\end{lem}
\begin{proof}
	Fix $\varepsilon>0$, $x\in M\setminus\C$, finite subsets $F\subset\Gamma$ and $E\subset M$. We will show that for every sufficiently small $\delta>0$, there exists $u\in\mathcal U(M)$ satisfying the following conditions:
\begin{itemize}
 \item[(a)] $\|u-\alpha_g(u)\|_\varphi<\delta$ for all $g\in F$, and $\|u\varphi-\varphi u\|<\delta$;
 \item[(b)] there exists $h\in\Gamma$ such that 
 \[
   \|\alpha_h(uxu^*)-uxu^*\|_\varphi>\|uxu^*-\varphi(uxu^*)\|_\varphi;
 \]
 \item[(c)] there exists $h\in\Gamma$ such that 
 \[
   \sum_{a,b\in E}\bigl|
     \varphi(\alpha_h(uau^*)ubu^*)-\varphi(uau^*)\varphi(ubu^*)
   \bigr|<\varepsilon.
 \]
\end{itemize}
If such a unitary $u$ exists, define an $\alpha$-cocycle $v_g:=u^*\alpha_g(u)$ for $g\in\Gamma$ and set $\psi:=u^*\varphi u$. Then $ \alpha^v$ preserves $\psi$, and by (a) the distance between $(1,\varphi)$ and $(v,\psi)$ is small. Moreover, (b) and (c) ensures that $(v,\psi)\in\mathcal U(x) \cap \mathcal V(E,\varepsilon)$.  Thus to prove the lemma, it suffices to construct such a unitary $u$.

For each $a\in M$, put $a^\circ:=a-\varphi(a)$.  
Since condition (b) is invariant under scalar rescaling of $x$, we may assume $\|x\|_\infty\leq 1/2$. Choose $\delta>0$ sufficiently small so that
\[
 8\delta\le\|x^\circ\|_\varphi^2,\qquad
 2|E|^2\max_{a\in E}\|2a\|_\infty^2 \,\delta<\varepsilon.
\]
By Lemma \ref{lem-cocycle3.6}, there exist $h\in\Gamma$ and $U\in\mathcal U(M^\omega)$ such that
\begin{itemize}
 \item $\varphi^\omega(U)=0$, \quad
   $\|\alpha_g^\omega(U)-U\|_{\varphi^\omega}<\delta$ for all $g\in F$, \quad
   $\|U\varphi^\omega-\varphi^\omega U\|<\delta$;
 \item $M$, $U$, and $\alpha_h^\omega(U)$ are $\ast$-free with respect to $\varphi^\omega$.
\end{itemize}
By $\|U\varphi^\omega-\varphi^\omega U\|<\delta$ and the free independence, for all $a\in M$ with $\|a\|_\infty\leq 1/2$, we have
\begin{align*}
 &\|Ua^\circ U^*-(UaU^*)^\circ\|_{\varphi^\omega}
   =|\!-\!\varphi(a)+\varphi^\omega(UaU^*)|
   <\delta ,\\
 &\|\alpha_h^\omega(UaU^*)-UaU^*\|_{\varphi^\omega}^2
	=\|\alpha_h^\omega(Ua^\circ U^*)-Ua^\circ U^*\|_{\varphi^\omega}^2
   =2\|Ua^\circ U^*\|_{\varphi^\omega}^2.
\end{align*}
Moreover, since $\|U\varphi^\omega-\varphi^\omega U\|<\delta$, we have
\begin{align*}
 \|a^\circ\|_\varphi^2
 &\le \|Ua^\circ U^*\|_{\varphi^\omega}^2
    +\|U\varphi^\omega U^*-\varphi^\omega\|\\
 &<(\|(UaU^*)^\circ\|_{\varphi^\omega}+\delta)^2+\delta
 <\|(UaU^*)^\circ\|_{\varphi}^2+4\delta.
\end{align*}
Since $8\delta\le\|x^\circ\|_\varphi^2$, taking $a=x$ implies
$4\delta<\|(UxU^*)^\circ\|_\varphi^2$.  
Therefore,
\[
 \|\alpha_h^\omega(UxU^*)-UxU^*\|_{\varphi^\omega}^2
 >2\|(UxU^*)^\circ\|_{\varphi^\omega}^2-4\delta
 >\|(UxU^*)^\circ\|_{\varphi^\omega}^2,
\]
which corresponds to condition (b).

For all $a,b\in M$ with $\|a\|_\infty,\|b\|_\infty\leq 1/2$, we have
\[
 |\varphi^\omega(\alpha_h^\omega[(UaU^*)^\circ-Ua^\circ U^*]b)|
 \le\|(UaU^*)^\circ-Ua^\circ U^*\|_{\varphi^\omega}\|b\|_{\varphi^\omega}
 <\delta,
\]
and similarly
$|\varphi^\omega(a[(UbU^*)^\circ-Ub^\circ U^*])|< \delta$.  
Hence,
\[
 |\varphi^\omega(\alpha_h^\omega((UaU^*)^\circ)(UbU^*)^\circ)|
 <|\varphi^\omega(\alpha_h^\omega(Ua^\circ U^*)Ub^\circ U^*)|+2\delta.
\]
By the free independence, the first term on the right-hand side vanishes.  
Using the equality $\varphi(\alpha_h(a^\circ)b^\circ) =\varphi(\alpha_h(a)b)-\varphi(a)\varphi(b)$,
we obtain
\[
 \sum_{a,b\in E}
   |\varphi^\omega(\alpha_h^\omega(UaU^*)UbU^*)
     -\varphi^\omega(UaU^*)\varphi^\omega(UbU^*)|
 <2\delta\, |E|^2\max_{a\in E}\|2 a\|_\infty^2<\varepsilon.
\]
This corresponds to condition (c).

Finally, representing $U\in M^\omega$ by a sequence in $M$, we can find a unitary element in $M$ satisfying the desired conditions.
\end{proof}

\section{Proof of main theorems}

Recall that for a group action $\alpha\colon \Gamma \actson M$, we are using the notation
\[
\Gamma_{\rm mod}= \{g\in \Gamma\mid \alpha_g = \Ad(u) \circ \sigma_t^{\varphi} \text{ for some } u \in \mathcal U(M),\ t \in \mathbb R \,\}.
\]
We distinguish two cases according to whether $\Gamma_{\mathrm{mod}}$ is trivial or not.

\subsection*{Case 1:  $\Gamma_{\rm mod}= \{e\}$.}

We prove the following theorem. This is a natural generalization of the result in \cite{MV23}.

\begin{thm}\label{thm-case1}
Let $\alpha \colon \Gamma \actson M$ be a state preserving outer action satisfying the following assumptions:
\begin{itemize}
	\item $\Gamma$ is amenable and countably infinite;
	\item $\Gamma_{\mathrm{mod}} = \{e\}$;
	\item $M$ is a type ${\rm III_1}$ factor with separable predual and has trivial bicentralizer.
\end{itemize}
Then both $\mathcal C_{\mathrm{erg}}(\alpha)$ and $\mathcal C_{\mathrm{wm}}(\alpha)$ are dense $G_\delta$ subsets of $\mathcal C_{\mathrm{state}}(\alpha)$.
\end{thm}

\begin{proof}
Take a countable subset $X = \{x_n\}_n \subset M \setminus \mathbb C$ that is dense in the $\sigma$-$\ast$strong topology.
By Lemma~\ref{lem-cocycle2.5}, we can write
\[
\mathcal C_{\mathrm{erg}}(\alpha) = \bigcap_{n\in \mathbb N} \mathcal U(x_n), \qquad
\mathcal C_{\mathrm{wm}}(\alpha) = \bigcap_{\substack{E\subset X\ \mathrm{finite}\\ n\in \mathbb N}} \mathcal V(E,1/n),
\]
as countable intersections of open sets.
If each $\mathcal U(x_n)$ and $\mathcal V(E,1/n)$ is dense in $\mathcal C_{\mathrm{state}}(\alpha)$, the conclusion follows from the Baire category theorem.
Hence, it suffices to show that $\mathcal U(x) \cap \mathcal V(E,\varepsilon) \subset \mathcal C_{\mathrm{state}}(\alpha)$ is dense for each $x\in M\setminus \mathbb C$, finite subset $E\subset M\setminus \mathbb C$, and $\varepsilon>0$.

Take an arbitrary $(w,\phi)\in \mathcal C_{\mathrm{state}}(\alpha)$. It is enough to approximate it by elements of $\mathcal U(x) \cap \mathcal V(E,\varepsilon)$.
Let $\gamma = \Ad(w)\circ \alpha$, and observe that if $v$ is an $\alpha$-cocycle, then $(vw^*)_g:=v_gw^*_g$ for $g\in \Gamma$ defines a $\gamma$-cocycle. Then it is easy to see that the map
\[
\mathcal C_{\mathrm{state}}(\alpha)\ni (v,\psi)\mapsto (vw^*,\psi)\in \mathcal C_{\mathrm{state}}(\gamma)
\]
is a homeomorphism. This map sends $(w,\phi)$ to $(1,\phi)$ and carries $\mathcal U(x)$, $ \mathcal V(E,\varepsilon)$ naturally to their counterparts for $\gamma$.
Therefore, if we can approximate $(1,\phi)$ with respect to $\gamma$, then we can approximate $(w,\phi)$ with respect to $\alpha$.
Thus, we may assume $w=1$. Then $\alpha\colon \Gamma \actson M$ is an action preserving $\phi\in S_{\mathrm{f}}(M)$.
By Lemma~\ref{lem-relative-bicentralizer}, we have $\mathrm{BC}(M\subset M\rtimes\Gamma,\phi)=\mathbb C$.
Hence Lemma~\ref{lem-cocycle3.5} applies, and the conclusion follows.
\end{proof}

\subsection*{Case 2: $\Gamma_{\rm mod}\neq \{e\}$.}

This case does not appear in \cite{MV23}, so additional arguments are required. For technical reasons, we assume $\Gamma = \mathbb Z$. Then, since $\{e\}\neq \Gamma_{\mathrm{mod}}\leq \mathbb Z$, there exists $p\geq 1$ such that $\Gamma_{\mathrm{mod}} = p\mathbb Z$.

\begin{thm}\label{thm-case2}
Let $M$ be a type ${\rm III_1}$ factor with separable predual, $\varphi\in M_\ast$ a faithful state, and $\alpha\colon \mathbb Z \actson (M,\varphi)$ a state preserving outer action such that $\Gamma_{\mathrm{mod}} = p\mathbb Z$ for some $p\geq 1$.
Then there exists a unitary $u\in \mathcal U(M)$ such that $\alpha^u$ is an ergodic state preserving action.
\end{thm}

\begin{proof}
	Since $p\in \Gamma_{\mathrm{mod}}$, there exist $u\in \mathcal U(M)$ and $t\in \mathbb R$ such that $\alpha_{p}=\Ad(u)\circ \sigma^\varphi_t$. Then $\Ad(u)$ preserves $\varphi$, so we have $u\in M_\varphi$. If $p=1$, then Theorem~\ref{thm-MV23} easily yields the conclusion. So we assume $p\geq 2$.  Set $\theta := \alpha_1$.

\begin{claim}
There exists $v\in \mathcal U(M_\varphi\cap M^\alpha)$ such that $(\Ad(v)\circ \theta)^{p^2}=\sigma^\varphi_{pt}$.
\end{claim}

\begin{proof}
Define $\beta:=\sigma^{\varphi}_{-t/p}\circ \theta$. Since $\theta$ preserves $\varphi$, it commutes with $\sigma^\varphi$, so that
\[
\beta^p = \sigma^\varphi_{-t} \circ \theta^p = \Ad(u).
\]
Since $\Gamma_{\mathrm{mod}} = p\mathbb Z$, $\beta^q$ is not inner for all $1\le q<p$. Hence, there exists an \emph{obstruction} (see \cite[Introduction]{Co75}) $\gamma\in \mathbb C$ with $|\gamma|=1$ such that
\[
\beta(u)=\sigma^{\varphi}_{-t/p}\circ \theta(u) = \gamma u,\quad \gamma^p=1.
\]
This implies $\theta(u)=\gamma\sigma^{\varphi}_{t/p}(u)=\gamma u$, and thus $\theta(u^p)=\gamma^p u^p=u^p$.
Hence $u^p$ is invariant under both $\theta$ and $\sigma^\varphi$.
Choose a unitary $v\in M_\varphi \cap M^\theta$ satisfying $v^{p^2}=u^p$.
Then we have
\[
(\Ad(v^*)\circ \theta)^{p^2}
= \Ad(v^{*p^2})\circ \theta^{p^2}
= \Ad(u^{*p})\circ (\Ad(u)\circ \sigma^\varphi_t)^p
= \sigma^\varphi_{pt}.
\]
Thus $v^*$ satisfies the required condition.
\end{proof}

By the claim, we may assume $\alpha_{p^2} = \sigma^\varphi_{pt}$.
Define $\beta :=\sigma^\varphi_{-t/p}\circ \theta$.
Then $\beta^p=\Ad(u)$ and $\beta^{p^2}=\mathrm{id}$, so $\beta$ defines an action of the finite group $\Lambda:=\mathbb Z/p^2\mathbb Z$. We again denote it by $\beta\colon \Lambda \actson M$. By Lemma~\ref{lem-fixed-point-typeIII} below, $M^\beta$ is a finite direct sum of type ${\rm III_1}$ factors, which we denote by
\[
M^\beta = N_1 \oplus \cdots \oplus N_n.
\]
By Theorem~\ref{thm-MV23}, there exist ergodic states $\psi_i\in (N_i)_\ast$ for all $i$. Define $\psi := \frac{1}{n}(\psi_1\oplus \cdots \oplus \psi_n) \in (M^\beta)_\ast$.
For any $T>0$, since $\sigma^\psi_T = \sigma^{\psi_1}_T \oplus \cdots \oplus \sigma^{\psi_n}_T$, and since each $\sigma^{\psi_i}_T$ is ergodic, we have
\[
(M^\beta)^{\sigma^\psi_T} = \C1_{N_1}\oplus \cdots \oplus \C 1_{N_n}= \mathcal Z(M^\beta).
\]
We next extend $\psi$ to $M$. 

Since $\Lambda=\mathbb Z/p^2\mathbb Z$ is a finite abelian group, we have the spectral decomposition
\[
M = \bigoplus_{\gamma \in \widehat{\Lambda}} M_\gamma,\quad \text{where }M_\gamma = \{x\in M \mid \beta_h(x)=\langle h, \gamma\rangle x \text{ for all }h\in \Lambda\}.
\]
Then there is a projection from $M$ onto $M_\gamma$ given by
\[
P_\gamma(x):=  \int_{\Lambda}\overline{\langle h,\gamma\rangle}\,\beta_h(x)\, dh.
\]
In particular, $M_e = M^\beta$ and
\[
E(x):=P_e(x) = \int_{\Lambda}\beta_h(x)\, dh
\]
is the canonical conditional expectation.
Clearly, $E\circ \beta_h =E= \beta_h\circ E$ for all $h\in \Lambda$.
Thus $\psi\circ E \in M_\ast$, which we again denote by $\psi$, is invariant under the $\Lambda$-action.

\begin{claim}
For any $T>0$, $M^{\sigma^\psi_T}$ is finite-dimensional.
\end{claim}

\begin{proof}
Since $\sigma^\psi$ commutes with $\beta$, it also commutes with each $P_\gamma$, hence acts on each $M_\gamma$.
Thus, we have
\[
M^{\sigma^\psi_T} = \bigoplus_{\gamma \in \widehat{\Lambda}} M_\gamma^{\sigma^\psi_T}.
\]
It suffices to show that each $M_\gamma^{\sigma^\psi_T}$ is finite-dimensional.

Write $\mathcal Z(M^\beta) = \sum_{i=1}^n \mathbb C e_i$, where each $e_i$ is a minimal projection. Fix $\gamma,i,j$ and take $x,y\in e_i M_\gamma^{\sigma_T^\psi} e_j$.
Then $y^*x$ is invariant under both $\beta$ and $\sigma^\psi_T$, so
\[
y^*x \in e_j (M^\beta )^{\sigma_T^\psi} e_j = \mathbb C e_j.
\]
In particular, by seeing the case $x=y$, we get $x^*x \in \mathbb C e_j$, and similarly $xx^*\in \mathbb C e_i$.
Hence, if $x\neq 0$, it is a scalar multiple of a partial isometry connecting $e_i$ and $e_j$. Then $y^* x\in \C e_j$ is equivalent to $y^* \in \mathbb C x^*$. We conclude that $e_i M_\gamma^{\sigma_T^\psi} e_j$ is at most one-dimensional. Therefore, by the direct sum decomposition
\[
M_\gamma^{\sigma_T^\psi} = \sum_{i,j=1}^n  e_i M_\gamma^{\sigma_T^\psi} e_j,
\]
we conclude that $M_\gamma^{\sigma_T^\psi}$ is finite-dimensional.
\end{proof}

Since $\beta = \theta \circ \sigma^\varphi_{-t/p}$ preserves $\varphi$, we have $\varphi \circ E = \varphi$, hence
\[
[D\psi:D\varphi]_t = [D\psi\circ E :D\varphi\circ E]_t \in M^\beta.
\]
We define
\[
u\colon \mathbb Z \to \mathcal U(M^\beta),\quad u_n := [D\psi:D\varphi]_{nt/p}.
\]
This is a cocycle for $\alpha$.
Indeed, by using the definition of the Connes cocycle,
\[
u_{n}\alpha_n(u_m)
= u_n \sigma^\varphi_{nt/p}\circ \beta^n(u_m)
= u_n \sigma^\varphi_{nt/p}(u_m)
= u_{n+m}.
\]
Thus $\alpha^u_n:=\Ad(u_n)\circ \alpha_n$ for $n\in \Z$ defines an action of $\mathbb Z$. It satisfies
\[
\alpha^u_n=\Ad(u_{n})\circ \sigma^\varphi_{nt/p}\circ \beta^{n} = \sigma^\psi_{nt/p}\circ \beta^{n}\quad \text{for all }n\in \Z.
\]
Hence $\alpha^u$ preserves $\psi$, and since $\beta$ vanishes on $p^2\mathbb Z$, we have $\alpha^u_{p^2n} =  \sigma^\psi_{npt}$ for all $n$.
This implies
\[
M^{\alpha^u} \subset M^{\sigma^\psi_{pt}}.
\]
By the previous claim, the right-hand side is finite-dimensional. Thus, $M^{\alpha^u}$ is finite-dimensional and contains a minimal projection. By Lemma~\ref{lem-ergodic-minimal} below, after modifying by another cocycle, we obtain an ergodic state preserving action. This completes the proof.
\end{proof}

The following lemmas were used in the above proof.

\begin{lem}\label{lem-ergodic-minimal}
Let $\alpha\colon \Gamma \actson M$ be an action of a discrete group on a $\sigma$-finite type ${\rm III}$ factor. If $M^\alpha$ contains a minimal projection, then there exists an $\alpha$-cocycle $v$ such that $\alpha^v$ is ergodic. If $\alpha$ is a state preserving action, then so is $\alpha^v$. 
\end{lem}

\begin{proof}
Take a minimal projection $p \in M^\alpha$.
Then the restricted action $\alpha^p \actson pMp$ is ergodic.
Since $M$ is a $\sigma$-finite type ${\rm III}$ factor, there exists a partial isometry $w$ satisfying $w^*w = 1$ and $ww^* = p$.
Define an $\alpha$-cocycle $v\colon \Gamma \to \mathcal U(M)$ by $v_g := w^* \alpha_g(w)$ for $g \in \Gamma$. Consider the $\ast$-isomorphism
\[
M \to pMp, \quad x \mapsto w x w^*,
\]
and transport the action $\alpha^p$ from $pMp$ to $M$. Then the resulting action  via the $\ast$-isomorphism coincides with $\alpha^v$. Hence $\alpha^v$ is ergodic.
If $\alpha$ preserves a faithful state $\varphi \in M_\ast$, then $\alpha^v$ preserves the faithful positive functional $\varphi(w\,\cdot\,w^*) \in M_\ast$.
\end{proof}

\begin{lem}\label{lem-fixed-point-typeIII}
Let $\alpha\colon \Lambda \actson M$ be an action of a finite group on a $\sigma$-finite type ${\rm III_1}$ factor $M$. Then $M^\alpha$ is a finite direct sum of type ${\rm III_1}$ factors.
\end{lem}

\begin{proof}
Let $\Lambda_{\mathrm{inn}} := \alpha^{-1}(\mathrm{Int}(M))$.
For each $g\in \Lambda_{\mathrm{inn}}$, fix $u_g \in \mathcal U(M)$ such that $\alpha_g = \Ad(u_g)$.
Since $M$ is a factor, each $u_g$ is unique up to a scalar multiple.
Using this uniqueness and the standard Fourier expansion argument, we have
\[
M'\cap (M\rtimes_\alpha \Lambda)
= \mathrm{span}\{ \lambda_g u_g^* \mid g\in \Lambda_{\mathrm{inn}}\}.
\]
In particular, this algebra is finite-dimensional.

\begin{comment}
Let $\langle M, M^\alpha \rangle$ be the basic construction for $M^\alpha \subset M$. Since the action of a finite group is integrable, there exists a normal surjective $\ast$-homomorphism
\[
M\rtimes_{\alpha}\Lambda \to \langle M, M^\alpha \rangle
\]
that restricts to the identity on $M$.
Hence there exists a central projection $z \in M\rtimes_{\alpha}\Lambda$ such that
\[
(M\rtimes_\alpha\Lambda)z \simeq \langle M, M^\alpha \rangle.
\]
It follows that $M'\cap \langle M, M^\alpha\rangle$ is isomorphic to the finite-dimensional algebra
\[
 (Mz)'\cap [(M\rtimes_\alpha \Lambda)z]=[M'\cap (M\rtimes_\alpha \Lambda)]z ,
\]
so it is finite-dimensional.
Applying the compression map by the modular conjugation on $L^2(M)$, we conclude that $(M^\alpha)'\cap M$ is finite-dimensional.
\end{comment}

Since $\Lambda$ is finite, $\alpha$ preserves some faithful state $\varphi\in M_\ast$.
Then $\alpha$ commutes with $\sigma^\varphi$, and hence there is a natural isomorphism on continuous cores:
\[
C_\varphi(M\rtimes_\alpha \Lambda)\simeq C_\varphi(M)\rtimes_{\widetilde{\alpha}} \Lambda,
\]
where $\widetilde{\alpha}\colon \Lambda \actson C_\varphi(M)$ is the canonical extension, acting trivially on $\mathbb R$. 
Since $M$ is a type ${\rm III_1}$ factor, for each $g\in \Lambda$, $\widetilde{\alpha}_g$ is inner if and only if $\alpha_g \in \mathrm{Mod}(M)$ \cite[Lemma XII.6.14]{Ta03}. Observe that  $\Lambda_{\mathrm{inn}} = \Lambda_{\mathrm{mod}}$ as $\Lambda$ is finite and $M$ is of type $\rm III_1$. Comparing coefficients in the Fourier expansion in the crossed product $C_\varphi(M)\rtimes_{\widetilde{\alpha}} \Lambda$, we obtain
\[
C_\varphi(M)'\cap (C_\varphi(M)\rtimes_{\widetilde{\alpha}} \Lambda)
= \mathrm{span}\{ \lambda_g u_g^* \mid g\in \Lambda_{\mathrm{inn}}\}
\subset M\rtimes_{\alpha} \Lambda.
\]
Thus we have $\mathcal Z(C_\varphi(M\rtimes \Lambda)) = \mathcal Z(M\rtimes \Lambda)$, which is finite-dimensional. Taking a minimal projection $z \in \mathcal Z(M\rtimes \Lambda)$, we find
\[
\mathcal Z(C_\varphi((M\rtimes \Lambda)z))
= \mathcal Z(C_\varphi(M\rtimes \Lambda))z
= \mathbb C z,
\]
so each $(M\rtimes \Lambda)z$ is a type ${\rm III_1}$ factor.
Therefore $M\rtimes_\alpha \Lambda$ is a finite direct sum of type ${\rm III_1}$ factors.

\begin{comment}
Considering again the surjection $M\rtimes_{\alpha}\Lambda \to \langle M, M^\alpha\rangle$, we see that $\langle M, M^\alpha\rangle$ is also a finite direct sum of type ${\rm III_1}$ factors.
Since this algebra is of type ${\rm III}$, it is anti-isomorphic to its commutant, and hence $\langle M, M^\alpha\rangle$ and $M^\alpha$ are $\ast$-isomorphic.
Consequently, $M^\alpha$ is a finite direct sum of type ${\rm III_1}$ factors.
\end{comment}

Let $\langle M, M^\alpha \rangle$ be the basic construction for $M^\alpha \subset M$. Since the action of a finite group is integrable, there exists a normal surjective $\ast$-homomorphism
\[
M\rtimes_{\alpha}\Lambda \to \langle M, M^\alpha \rangle
\]
that restricts to the identity on $M$.
Hence there exists a central projection $w \in M\rtimes_{\alpha}\Lambda$ such that
\[
(M\rtimes_\alpha\Lambda)w \simeq \langle M, M^\alpha \rangle.
\]
It follows that $\langle M, M^\alpha\rangle$ is also a finite direct sum of type ${\rm III_1}$ factors.
Since this algebra is of type ${\rm III}$, it is anti-isomorphic to its commutant, and hence $\langle M, M^\alpha\rangle$ and $M^\alpha$ are $\ast$-isomorphic.
Consequently, $M^\alpha$ is a finite direct sum of type ${\rm III_1}$ factors.
\end{proof}

\subsection*{Proof of main theorems}

Now we prove main theorems. 

\begin{proof}[{\bf Proof of Theorem \ref{thmA}}]
(1)$\Rightarrow$(2) 
Consider the $\mathbb Z$-action induced by $\theta$ that is a state preserving outer action.
If $\mathbb Z_{\mathrm{mod}} = \{e\}$, then the conclusion follows from Theorem \ref{thm-case1}. If $\mathbb Z_{\mathrm{mod}} \neq \{e\}$, we can use Theorem \ref{thm-case2}.

(2)$\Rightarrow$(1) This follows from the same argument as in \cite[Theorem 3.2]{MV23}.
\end{proof}

\begin{proof}[{\bf Proof of Corollary \ref{corB}}]
By \cite[Lemma~5.1]{HI22}, there exists $u\in\mathcal U(M)$ such that $\Ad(u)\circ\theta$ is a  state preserving action. Then we can apply Theorem \ref{thmA}.
\end{proof}

\begin{proof}[{\bf Proof of Theorem \ref{thmC}}]
	This is immediate from Theorem \ref{thm-case1}.
\end{proof}

\begin{proof}[{\bf Proof of Corollary \ref{corD}}]
By \cite[Theorem~E(3)]{Ma23}, the tensor product $M\overline\otimes N$ has trivial bicentralizer. Suppose that $\alpha_g\otimes\id \in \mathrm{Mod}(M\overline\otimes N)$  for some $g\in\Gamma$. Then there exist $t\in\mathbb R$, faithful states $\varphi\in M_\ast$, $\psi\in N_\ast$, and a unitary $u\in\mathcal U(M\overline\otimes N)$ such that
\[
\alpha_g\otimes\id = \Ad(u)\circ(\sigma_t^\varphi\otimes\sigma_t^\psi).
\]
It follows that $(\alpha_g\circ\sigma_{-t}^\varphi)\otimes\sigma_{-t}^\psi \in \mathrm{Int}(M\overline\otimes N)$,
hence $\alpha_g\circ\sigma_{-t}^\varphi\in \mathrm{Int}(M)$ and $\sigma_{-t}^\psi \in \mathrm{Int}(N)$. Since $N$ is a type $\mathrm{III}_1$ factor, we must have $t=0$. This implies $\alpha_g\in\mathrm{Int}(M)$ and $g=e$. Thus, for the action $\alpha\otimes\id$, we have $\Gamma_{\mathrm{mod}}=\{e\}$. The conclusion then follows from Theorem \ref{thmC}.
\end{proof}

\section{Further results}
\label{Further results}

In this last section, we collect several related results.

\subsection*{The case of type $\rm III_\lambda$ factors}

For a state preserving action $\alpha\colon \Gamma \actson (M,\varphi)$, we denote by $\mathcal C(\alpha,\varphi)$ the set of all $\alpha$-cocycles taking values in $M_\varphi$. For each $v\in \mathcal C(\alpha,\varphi)$, the perturbed action $\alpha^v$ is again $\varphi$-preserving. We define $\mathcal C_{\mathrm{erg}}(\alpha,\varphi)$ and $\mathcal C_{\mathrm{wm}}(\alpha,\varphi)$ to be the subsets of $\mathcal C(\alpha,\varphi)$ consisting of ergodic and weakly mixing cocycles, respectively. If $M$ has separable predual and $\Gamma$ is countable, then $\mathcal C(\alpha,\varphi)$ is a Polish space. We also note that if $M$ is a type $\mathrm{II}_1$ factor and $\varphi$ is a trace, then $\mathcal C(\alpha,\varphi) = \mathcal C(\alpha)$.

The following proposition is proved by the same argument as Theorem~\ref{thm-case1}.  
Its proof is simpler, since under the stronger irreducibility assumption below, 
the argument of \cite{MV23} applies directly.

\begin{prop}\label{prop-case3}
	Let $\alpha \colon \Gamma \actson (M,\varphi)$ be a state preserving outer action 
	satisfying the following assumptions:
	\begin{itemize}
		\item $\Gamma$ is amenable and countably infinite;
		\item $M$ is a diffuse factor with separable predual and
		\[
		M_\varphi' \cap (M\rtimes_\alpha \Gamma)=\mathbb C.
		\]
	\end{itemize}
	Then both $\mathcal C_{\mathrm{erg}}(\alpha,\varphi)$ and 
	$\mathcal C_{\mathrm{wm}}(\alpha,\varphi)$ are dense $G_\delta$ subsets of 
	$\mathcal C(\alpha,\varphi)$.
\end{prop}

\begin{proof}
	Since $M_\varphi\subset (M^\omega)_{\varphi^\omega}$, the last assumption implies
	\[
	\mathrm{BC}(M\subset M\rtimes_\alpha\Gamma,\varphi)
	=
	\bigl((M^\omega)_{\varphi^\omega}\bigr)'
	\cap (M\rtimes_\alpha\Gamma)
	=\mathbb C.
	\]
	Moreover, by the last assumption and \cite{HI14}
	(see Remark \ref{rem-free-independence1}),
	the unitaries appearing in the conclusion of Proposition~\ref{prop-free-independence1} 
	may be chosen inside $(M_\varphi)^\omega$.  
	Thus in Lemma~\ref{lem-free-independence2}, if we set
	\[
	P_0 := \mathrm{W}^*\{\alpha_g^\omega (V M_\varphi V^*)\}_{g\in \Gamma},
	\]
	then $P_0 \subset (M_\varphi)^\omega$.  
	Since $M_\varphi$ is diffuse, we may apply Lemma~\ref{lem-amenable-strongly-ergodic} 
	to the free product $\ast_{g\in \Gamma} (M_\varphi,\varphi)$, and therefore the unitary 
	$U$ obtained in the conclusion of Lemma~\ref{lem-cocycle3.6} can be chosen from 
	$\mathcal U\bigl( (M_\varphi)^\omega \bigr)$.
	
	Consequently, in Lemma~\ref{lem-cocycle3.5}, when approximating cocycles by elements of 
	$\mathcal U(x)\cap \mathcal V(E,\varepsilon)$, one may approximate them by cocycles 
	with values in $M_\varphi$.  
	Then the conclusion follows by the same proof as in Theorem~\ref{thm-case1}.
\end{proof}

If $M$ is a type $\mathrm{III}_\lambda$ factor ($0<\lambda<1$) and $\varphi$ is a 
$\lambda$-trace, then $M_\varphi' \cap M=\mathbb C$ holds automatically.  
In this case, the condition $\Gamma_{\mathrm{mod}}=\{e\}$ also implies the 
irreducibility assumption in Proposition~\ref{prop-case3}, as we will see below.  
Thus we obtain the following.

\begin{cor}
	Let $M$ be a type $\mathrm{III}_\lambda$ factor $(0<\lambda<1)$ with separable predual, and let 
	$\varphi\in M_\ast$ be a faithful state such that $\sigma^\varphi$ has period 
	$T$, where $T=-2\pi/\log(\lambda)$.  
	Let $\alpha \colon \Gamma \actson (M,\varphi)$ be a state preserving outer action 
	satisfying the following assumptions:
	\begin{itemize}
		\item $\Gamma$ is amenable and countably infinite;
		\item $\Gamma_{\mathrm{mod}} = \{e\}$.
	\end{itemize}
	Then both $\mathcal C_{\mathrm{erg}}(\alpha,\varphi)$ and 
	$\mathcal C_{\mathrm{wm}}(\alpha,\varphi)$ are dense $G_\delta$ subsets 
	of $\mathcal C(\alpha,\varphi)$.
\end{cor}

\begin{proof}
	We have only to show that $M_\varphi' \cap (M\rtimes_\alpha\Gamma)=\mathbb C$. 
	Take $x\in M_\varphi' \cap (M\rtimes_\alpha\Gamma)$ with the Fourier decomposition $x=\sum_{g\in \Gamma} x_g \lambda_g$. 
	We have
	\[
	ax_g=x_g\alpha_g(a),\qquad a\in M_\varphi.
	\]
	Since $\alpha_g(M_\varphi)=M_\varphi$ and $M_\varphi'\cap M=\mathbb C$, it follows that $x_gx_g^*,\,x_g^*x_g\in\mathbb C$. 
	Thus, if $x_g\neq0$, then $x_g$ is a scalar multiple of a unitary
	$u_g\in\mathcal U(M)$ and
	\[
	\operatorname{Ad}(u_g)\circ\alpha_g|_{M_\varphi}
	=\operatorname{id}_{M_\varphi}.
	\]
	As $\varphi$ is a $\lambda$-trace, this implies that $\Ad(u_g)\circ\alpha_g=\sigma_t^\varphi$
	for some $t\in\mathbb R$. Since
	$\Gamma_{\mathrm{mod}}=\{e\}$, we have $g=e$.
	
	Therefore $x_g=0$ for every $g\neq e$, while
	$x_e\in M_\varphi'\cap M=\mathbb C$. Hence $x\in\mathbb C$.
	The conclusion follows from Proposition~\ref{prop-case3}.
\end{proof}

\begin{rem}\upshape
	The assumption $\Gamma_{\mathrm{mod}}=\{e\}$ is essential.
	For example, if $\Gamma=\mathbb Z=\Gamma_{\mathrm{mod}}$, then
	$\mathcal C_{\mathrm{erg}}(\alpha,\varphi)=\emptyset$. Indeed,  
	write $\alpha_1=\theta$, and express 
	$\theta=\operatorname{Ad}(u)\circ\sigma_t^\varphi$ for some 
	$u\in\mathcal U(M)$ and $t\in\mathbb R$.  
	Since $\theta$ preserves $\varphi$, we have $u\in M_\varphi$.
	For any $v\in\mathcal C(\alpha,\varphi)$, we have
	\[
	(\alpha^v)_1
	=
	\operatorname{Ad}(v_1u)\circ\sigma_t^\varphi,
	\qquad
	v_1u\in M_\varphi.
	\]
	If $v_1u\in\mathbb C$, then $M^{\alpha^v}$ contains the type 
	$\mathrm{II}_1$ factor $M_\varphi$.  
	If $v_1u\notin\mathbb C$, then $v_1u\in M^{\alpha^v}$, and hence 
	$M^{\alpha^v}\neq\mathbb C$.  
	Thus $\alpha^v$ is not ergodic.
\end{rem}

\subsection*{Trace scaling actions}

\begin{prop}
Let $(M,\Tr)$ be a type $\mathrm{II}_\infty$ factor with trace.  
If $\alpha\colon \mathbb Z \actson M$ is a trace scaling action, then for any 
$\alpha$-cocycle $v$, the perturbed action $\alpha^v$ is not ergodic.
\end{prop}

\begin{proof}
Since $M\rtimes_\alpha \mathbb Z$ is a type $\mathrm{III}_\lambda$ factor for some $0<\lambda <1$, this crossed product gives a discrete decomposition.  By the uniqueness of the discrete 
decomposition, the fixed point algebra $M^\alpha$ is isomorphic to 
$M\rtimes_\alpha \mathbb Z$, and hence $\alpha$ is not ergodic.

Next, $\beta:=\alpha^v$ is also trace scaling, and 
$M\rtimes_\alpha \mathbb Z \simeq M\rtimes_\beta \mathbb Z$.  
By the uniqueness of the discrete decomposition again, $\alpha$ and $\beta$ are 
conjugate.  In particular, their fixed point algebras are isomorphic, and thus 
$\beta$ is not ergodic.
\end{proof}

\subsection*{Free product extension}

The following proposition is an adaptation of \cite[Proposition 3.4]{MV23}. It includes the case of actions of free groups.

\begin{prop}
	Let $\Lambda $ be a nontrivial countable group and put $\Gamma := \Z \ast \Lambda$. Let $\alpha\colon \Gamma \actson M$ be an outer action on a type $\rm III_1$ factor with separable predual and with trivial bicentralizer. If $\alpha$ is state preserving on $\Z$, then $\alpha$ admits an ergodic cocycle.
\end{prop}
\begin{proof}
	Put $\beta:=\alpha|_\Z$. Then by Theorem \ref{thmA}, there exists a $\beta$-cocycle $v$ such that $\beta^v$ is ergodic. Then there exists a group homomorphism
	$$\Gamma\to \mathcal U(M)\rtimes \Aut(M);\quad g\mapsto (u_g,\alpha_g)$$
such that $u_g =v_g$ for $g\in \Z$ and $u_g =1$ for $g\in \Lambda$. It is an $\alpha$-cocycle extending $v$, and we conclude that $\alpha^u$ is ergodic. Note that $\alpha^u$ is not necessarily a state preserving action, even if $\alpha$ is state preserving on $\Gamma$.
\end{proof}

\subsection*{Approximate vanishing of state preserving cocycles}

The following is a theorem corresponding to \cite[Theorem~4.1]{PSV18}. See also \cite[Theorem 3.3]{Ma18}.

\begin{thm}\label{thm-cohomology-vanishing}
Let $M$ be a diffuse factor with separable predual, and let 
$\alpha\colon \Gamma \actson (M,\varphi)$ be a state preserving outer action.  
Assume that $\Gamma$ is a countably infinite amenable group and 
$\mathrm{BC}(M\subset M\rtimes \Gamma,\varphi)=\mathbb C$.
\begin{enumerate}

\item 
The inclusion $(M^\omega)^{\alpha^\omega}_{\varphi^\omega}\subset M^\omega$ is 
irreducible.  
In particular, $(M^\omega)^{\alpha^\omega}_{\varphi^\omega}$ is a type $\mathrm{II}_1$ factor, and $\alpha$ is not strongly ergodic.

\item 
Assume that $M$ is a type $\mathrm{III}_1$ factor.  
Then for any $(v,\psi) \in \mathcal C_{\mathrm{state}}(\alpha)$, the following 
holds:  
for any finite subset $F\subset \Gamma$ and any $\varepsilon>0$, there exists 
a unitary $u\in\mathcal U(M)$ such that
\[
  \|u\alpha_g(u^*) - v_g\|_{2,\varphi} < \varepsilon
  \quad (g\in F),
  \qquad 
  \|\psi u- u\varphi \| < \varepsilon.
\]
Equivalently, there exists a unitary $U\in M^\omega$ such that
\[
  U\alpha_g^\omega(U^*) = v_g \quad (g\in\Gamma),
  \qquad 
  \psi^\omega = U\varphi^\omega U^*.
\]

\item 
Assume that
$M_\varphi' \cap (M\rtimes_\alpha\Gamma)=\mathbb C$.
Then the inclusion
$((M_\varphi)^\omega)^{\alpha^\omega}\subset M^\omega$
is irreducible. Moreover, for any $v\in\mathcal C(\alpha,\varphi)$
the following holds: for any finite subset $F\subset\Gamma$ and any
$\varepsilon>0$, there exists a unitary $u\in\mathcal U(M_\varphi)$
such that
\[
\|u\alpha_g(u^*)-v_g\|_{2,\varphi}<\varepsilon
\qquad (g\in F).
\]
Equivalently, there exists a unitary $U\in(M_\varphi)^\omega$ such that
\[
U\alpha_g^\omega(U^*)=v_g
\qquad (g\in\Gamma).
\]

\end{enumerate}
\end{thm}

\begin{proof}
(1)  
For each von Neumann subalgebra $B \subset (M^\omega)_{\varphi^\omega}$, the 
inclusion $B' \cap M^\omega \subset M^\omega$ admits a unique $\varphi^\omega$-preserving conditional expectation, which we denote by $E_{B' \cap M^\omega}$.  To prove item (1), it suffices 
to show $E_{B'\cap M^\omega} = \varphi^\omega$ when $B=(M^\omega)_{\varphi^\omega}^{\alpha^\omega}$. 
For this, it is enough to show that for any $x \in M^\omega$ there exists a 
von Neumann subalgebra $B \subset (M^\omega)^{\alpha^\omega}_{\varphi^\omega}$ 
such that
\[
  E_{B'\cap M^\omega}(x) = \varphi^\omega(x).
\]
We fix an arbitrary $x=(x_n)_\omega \in (M^\omega)_1$.  
Choose an increasing sequence of finite subsets $(F_n)_n$ of $\Gamma$ with 
$\Gamma=\bigcup_n F_n$, and a sequence $(d_n)_n$ in $\mathbb N$ such that 
$d_n \to \infty$.  
By Lemma~\ref{lem4.4} below, for each $n$ there exists a $d_n$-dimensional abelian von 
Neumann subalgebra $A_n \subset M$ satisfying:

\begin{itemize}

\item[(a)]  
For each minimal projection $p \in A_n$,
\[
  \varphi(p)=\frac{1}{d_n},\qquad 
  \|p\varphi - \varphi p\| < \frac{1}{n d_n},\qquad 
  \|\alpha_g(p)-p\|_2 < \frac{1}{n d_n}
  \quad\text{for all } g\in F_n.
\]

\item[(b)]  
$  \| E_{A_n' \cap M}(x_n) - \varphi(x_n) \|_2 < {4}/{\sqrt{d_n}}.
$
\end{itemize}
For each $n$, write $A_n = \sum_{i=1}^{d_n} \mathbb C e_i^n$, where each $e_i^n$ is a minimal projection. Define
\[
  \varphi_n := \sum_{i=1}^{d_n} e_i^n \, \varphi \, e_i^n \in M_\ast .
\]
This is a faithful normal state, and by (a) we have 
$\|\varphi - \varphi_n\| < 1/n$.  
Hence the state $(\varphi_n)_\omega$ on $M^\omega$ given by the sequence $(\varphi_n)_n$ agrees with $\varphi^\omega$.

Since $A_n \subset M_{\varphi_n}$, general facts about ultraproduct von Neumann 
algebras imply the followings.

\begin{itemize}

\item  
If $B$ denotes the ultraproduct 
$\prod^\omega (A_n,\varphi_n)$, then there is a natural inclusion 
$B \subset M^\omega$.

\item  
The inclusion $B' \cap M^\omega \subset M^\omega$ admits a unique 
$(\varphi_n)_\omega$-preserving conditional expectation, given by
\[
  E_{B' \cap M^\omega}((y_n)_\omega)
  = (\, E_{A_n' \cap M}(y_n) \,)_\omega,
  \qquad \text{for }(y_n)_\omega \in M^\omega.
\]

\end{itemize}
By (a) and the identity $(\varphi_n)_\omega = \varphi^\omega$, we have $B \subset (M^\omega)^{\alpha^\omega}_{\varphi^\omega}$.  
Then the above $E_{B' \cap M^\omega}$ is the unique $\varphi^\omega$-preserving  conditional expectation, hence agrees with the notation at the beginning of the proof.  Finally, by (b),
\[
  E_{B'\cap M^\omega}((x_n)_\omega) 
  = \varphi^\omega((x_n)_\omega),
\]
which is exactly the desired condition.

(2)  
Let $\widetilde M := \M_{2}(M)=M\otimes \M_2$, and equip it with the 
action $\widetilde\alpha_g := \alpha_g\otimes\id$ for $g\in \Gamma$.  
Let $(e_{ij})_{1\le i,j\le 2}$ be the matrix units coming from 
$\M_2 \subset \widetilde M$.  
For $g\in\Gamma$, define 
\[
  \widetilde v_g := e_{11} + v_g e_{22}.
\]
This defines a $\widetilde\alpha$-cocycle and the corresponding perturbed action is given by
\[
  \widetilde{\alpha}_g^{\widetilde v}\left(
  \begin{bmatrix}
    x_{11} & x_{12} \\
    x_{21} & x_{22}
  \end{bmatrix}
\right)
  =
  \begin{bmatrix}
    \alpha_g(x_{11}) 
      & \alpha_g(x_{12}) v_g^* \\
    v_g \alpha_g(x_{21}) 
      & v_g \alpha_g(x_{22}) v_g^*
  \end{bmatrix}.
\]
Consider the balanced weight $\phi := \varphi \oplus \psi$ on $\widetilde M$.  
Then $\widetilde{\alpha}^{\widetilde v}$ preserves $\phi$.  Observe that
\[
\mathrm{BC}\bigl(
\widetilde M\subset
\widetilde M\rtimes_{\widetilde\alpha^{\widetilde v}}\Gamma,\phi
\bigr)=\mathbb C.
\]
Indeed, since $\widetilde M$ is a type $\mathrm{III}_1$ factor, the triviality of the relative
bicentralizer does not depend on the choice of a faithful state. Hence, using
\[
\widetilde M\rtimes_{\widetilde\alpha^{\widetilde v}}\Gamma
\simeq (M\rtimes_\alpha\Gamma)\overline\otimes \mathbb M_2
\]
and computing with $\varphi\otimes\operatorname{tr}_2$, we obtain the desired triviality. 
Then, by item (1), the fixed point algebra of $(\widetilde{\alpha}^{\widetilde v})^\omega$ inside $(\widetilde M^\omega)_{\phi^\omega}$ is a type $\mathrm{II}_1$ factor.

The projections $e_{11},e_{22}$ belong to this fixed point algebra, and since 
they have the same value under $\phi$, they are equivalent in the fixed point algebra.  
Taking the $(2,1)$-entry of a partial isometry connecting them, we obtain 
$U\in \mathcal U(M^\omega)$ satisfying
\[
  v_g \alpha_g(U) = U \quad (g\in\Gamma), 
  \qquad 
  [D\psi : D\varphi]_t\, \sigma_t^{\varphi^\omega}(U)=U 
  \quad (t\in\mathbb R).
\]
The first equation yields the cocycle identity.  
The second implies
\[
  [D (U^* \psi^\omega U) : D\varphi^\omega]_t
   = U^* [D\psi^\omega : D\varphi^\omega]_t\, 
      \sigma_t^{\varphi^\omega}(U)
   = 1, \qquad t\in\mathbb R,
\]
and therefore $U^* \psi^\omega U = \varphi^\omega$.

(3) This is obtained simply by checking that, by the last assertion of
Lemma~\ref{lem4.4} below, the objects chosen in the above proofs can be
taken inside $M_\varphi$; the details are omitted.
\end{proof}

\begin{lem}\label{lem4.4}
Let $\Gamma \actson^\alpha (M,\varphi)$ be the action appearing in 
Theorem~\ref{thm-cohomology-vanishing}.  
For any finite subsets $F_0\subset\Gamma$, $X_0\subset M$, any $\delta>0$, and 
any $d\in\mathbb N$, there exists a $d$-dimensional abelian von Neumann 
subalgebra $A_0 \subset M$ such that:

\begin{itemize}

\item 
For each minimal projection $p \in A_0$,
\[
  \varphi(p)=\frac{1}{d},\qquad 
  \|p\varphi - \varphi p\| < \frac{\delta}{d},\qquad 
  \|\alpha_g(p) - p\|_{2,\varphi} < \frac{\delta}{d}
  \quad (g \in F_0).
\]

\item 
For all $x \in X_0$,
\[
  \|E_{A_0' \cap M}(x) - \varphi(x)\|_{2,\varphi}
  < 2 \|x - \varphi(x)\|_{2,\varphi}\, d^{-1/2}.
\]

\end{itemize}
Moreover, if $M_\varphi' \cap (M\rtimes_\alpha\Gamma)=\mathbb C$,
then $A_0$ may be chosen inside $M_\varphi$.
\end{lem}

\begin{proof}
Let $\beta \colon \Gamma \actson (P,\psi) := \ast_\Gamma(M,\varphi)$ be the free 
Bernoulli action.  
By Remark~\ref{rem-free-Bernoulli-diffuse}, 
$(P^\omega)^{\beta^\omega}_{\psi^\omega}$ is diffuse.  
Hence we may choose a $d$-dimensional abelian von Neumann subalgebra 
$B_0 \subset (P^\omega)^{\beta^\omega}_{\psi^\omega}$ whose minimal projections 
all have $\psi^\omega$-value $1/d$.  
Since $P$ is diffuse, by choosing representing sequences appropriately we obtain 
a $d$-dimensional abelian subalgebra $A_0 \subset P$ such that each minimal 
projection $p \in A_0$ satisfies
\[
  \psi(p) = \frac{1}{d},\qquad 
  \|\beta_g(p) - p\|_{2,\psi} < \frac{\delta}{d}\ (g\in F_0),\qquad 
  \|p\psi - \psi p\| < \frac{\delta}{d}.
\]
Next, by Lemma~\ref{lem-free-independence2}, we embed the free Bernoulli action into 
$M^\omega$ in such a way that
\[
  \alpha^\omega \colon \Gamma \actson M \ast P \subset M^\omega,
  \qquad 
  \alpha^\omega|_P = \beta.
\]
Since the subalgebra $A_0 \subset P$ sits freely from $M$, the free independence implies
\[
  \| E_{A_0' \cap M^\omega}(x) - \varphi(x)\|_{2,\varphi^\omega}
  = \|x - \varphi(x)\|_{2,\varphi^\omega}\, d^{-1/2}
  \qquad (x\in M).
\]
Moreover, because $\alpha^\omega|_P = \beta$, $\varphi^\omega|_P=\psi$ and $P \subset M^\omega $ is with $\varphi^\omega$-preserving expectation, if we regard $A_0$ as a subalgebra in $M^\omega$, then each minimal projection $p\in A_0$ satisfies
\[
  \varphi^\omega(p)=\frac{1}{d},\qquad
  \|\alpha_g^\omega(p) - p\|_{2,\varphi^\omega} < \frac{\delta}{d},\qquad
  \|p\varphi^\omega - \varphi^\omega p\| < \frac{\delta}{d}.
\]
Since $M$ is diffuse, we may again choose $A_0$ appropriately from $M$ itself, 
yielding the desired conclusion.

Finally, assume that
$M_\varphi' \cap (M\rtimes_\alpha\Gamma)=\mathbb C$.
Then, exactly as in the proof of Proposition~\ref{prop-case3}, we may
carry out the above argument inside $(M_\varphi)^\omega$.
Thus $A_0$ may be taken inside $M_\varphi$.
\end{proof}

\small{

}


\begin{thebibliography}{AHHM18}


\bibitem[AH12]{AH12} H. Ando and U. Haagerup, \textit{Ultraproducts of von Neumann algebras.} J. Funct. Anal. {\bf 266} (2014), 6842--6913.

\bibitem[AHHM18]{AHHM18} H. Ando, U. Haagerup, C. Houdayer, and A. Marrakchi, \textit{Structure of bicentralizer algebras and inclusions of type III factors.} Math. Ann. {\bf 376} (2020), no. 3-4, 1145--1194.

%\bibitem[BO08]{BO08} N. P. Brown and N. Ozawa, \textit{$C^{\ast}$-algebras and finite-dimensional approximations}. Graduate Studies in Mathematics, 88. American Mathematical Society, Providence, RI, 2008.

%\bibitem[Co72]{Co72} A.~Connes, \textit{Une classification des facteurs de type $\rm III$.} Ann. Sci. \'Ecole Norm. Sup. (4) {\bf 6} (1973), 133--252.

\bibitem[Co75]{Co75} A.~Connes, \emph{Outer conjugacy classes of automorphisms of factors.} Ann.\ Sci.\ \'Ecole Norm.\ Sup.\ (4) \textbf{8} (1975), 383--419.

%\bibitem[Co75]{Co75} A.~Connes, \textit{Classification of injective factors.} Ann. of Math. (2) {\bf 104} (1976), no. 1, 73--115.

%\bibitem[CT76]{CT76} A. Connes and M. Takesaki, \textit{The flow of weights on factors of type $\rm III$.} {T\^{o}hoku} Math. J. (2) {\bf 29} (1977), no. 4, 473--575.

%\bibitem[Ha77a]{Ha77a} U. Haagerup, \textit{Operator valued weights in von Neumann algebras, I.} J. Funct. Anal. {\bf 32} (1979), 175--206.

%\bibitem[Ha77b]{Ha77b} U. Haagerup, \textit{Operator valued weights in von Neumann algebras, II.} J. Funct. Anal. {\bf 33} (1979), 339--361.

%\bibitem[Ha85]{Ha85} U. Haagerup \textit{Connes' bicentralizer problem and uniqueness of the injective factor of type $\rm III_1$.} Acta Math. {\bf 158} (1987), no. 1-2, 95--148.

\bibitem[HI14]{HI14} C.~Houdayer and Y. Isono, \textit{Free independence in ultraproduct von Neumann algebras and applications.} J. London Math. Soc. {\bf 92} (2015), 163--177.

%\bibitem[HI15a]{HI15} C.~Houdayer and Y. Isono, \textit{Unique prime factorization and bicentralizer problem for a class of type $\rm III$ factors}. Adv. Math. {\bf 305} (2017), 402--455.

\bibitem[HI22]{HI22} C.~Houdayer and Y. Isono, \textit{Pointwise inner automorphisms of almost periodic factors.} Selecta Math. (N.S.) {\bf 30} (2024), 58.

\bibitem[HM26]{HM26} C. Houdayer and A. Marrakchi, \textit{The classification of flows on $\mathrm{II}_1$ factors and Connes' bicentralizer problem.} Preprint, 2026. \texttt{arXiv:2609.11462}

\bibitem[Is23]{Is23} Y. Isono, \textit{Haagerup and St\o rmer's conjecture on pointwise inner automorphisms.} Compos. Math. {\bf 160} (2024), 2480--2495.  

\bibitem[KST89]{KST89} Y. Kawahigashi, C.E. Sutherland, and M. Takesaki, \textit{The structure of the automorphism group of an injective factor and the cocycle conjugacy of discrete abelian group actions.} Acta Math. {\bf 169} (1992), 105--130.

\bibitem[Lo78]{Lo78} R. Longo, \textit{Notes on algebraic invariants for noncommutative dynamical systems.} Comm. Math. Phys. { \bf 69} (1979), no. 3, 195--207

\bibitem[Ma18]{Ma18} A. Marrakchi, \textit{Fullness of crossed products of factors by discrete groups.} Proc. Roy. Soc. Edinburgh Sect. A {\bf 150} (2020), no. 5, 2368--2378.


\bibitem[Ma23]{Ma23} A. Marrakchi, \textit{Kadison's problem for type $\rm III$ subfactors and the bicentralizer conjecture}. Invent. Math. {\bf 239} (2025), 79--163.

\bibitem[MV23]{MV23} A. Marrakchi and S. Vaes, \textit{Ergodic states on type $\rm III_1$ factors and ergodic actions.} J. Reine Angew. Math. {\bf 809} (2024), 247--260.

\bibitem[Oc85]{Oc85} A. Ocneanu, \textit{Actions of discrete amenable groups on von Neumann algebras.} Lecture Notes in Mathematics, {\bf 1138}. Springer-Verlag, Berlin, 1985. iv+115 pp.

%\bibitem[Po81]{Po81} S.~Popa, \textit{On a problem of R. V. Kadison on maximal abelian $\ast$-subalgebras in factors.} Invent. Math. {\bf 65} (1981), no. 2, 269--281.

%\bibitem[Po95]{Po95} S.~Popa, \textit{Classification of subfactors and their endomorphisms.} CBMS Regional Conf. Ser. in Math., {\bf 86}.

\bibitem[Po95]{Po95} S. Popa, \textit{Free-independent sequences in type ${\rm II_1}$ factors and related problems.} Recent advances in operator algebras (Orl\'eans, 1992). Ast\'erisque No. {\bf 232} (1995), 187--202.

%\bibitem[Po01]{Po01} S.~Popa, \textit{On a class of type $\rm II_1$ factors with Betti numbers invariants.} Ann.\ of Math.\ {\bf 163} (2006), 809--899.

%\bibitem[Po03]{Po03} S.~Popa, \textit{Strong rigidity of $\rm II_1$ factors arising from malleable actions of w-rigid groups~$\rm I$.} Invent.\ Math.\ {\bf 165} (2006), 369--408.

%\bibitem[Po05a]{Po05a} S.~Popa, \textit{Cocycle and orbit equivalence superrigidity for malleable actions of w-rigid groups.} Invent. Math. {\bf 170} (2007), no. 2, 243--295.

%\bibitem[Po06a]{Po06a} S. Popa, \textit{On the superrigidity of malleable actions with spectral gap.} J. Amer. Math. Soc. {\bf 21} (2008), no. 4, 981--1000.

%\bibitem[Po20]{Po20} S. Popa, \textit{On ergodic embeddings of factors.} Comm. Math. Phys. {\bf 384} (2021), no. 2, 971--996.


\bibitem[PSV18]{PSV18} S. Popa, D. Shlyakhtenko and S. Vaes, \textit{Classification of regular subalgebras of the hyperfinite $\rm II_1$ factor.} J. Math. Pures Appl. (9) {\bf 140} (2020), 280--308.

\bibitem[Ta03]{Ta03} M.~Takesaki, \textit{Theory of operator algebras $\rm II$}. Encyclopedia of Mathematical Sciences, 125. Operator Algebras and Non-commutative Geometry, {\bf 5}. Springer-Verlag, Berlin, 2002.


\end{thebibliography}
\end{document}